%% file: 0-main.tex
\documentclass[11pt,letterpaper]{article}

\newcommand{\blind}{0}

\usepackage[top=1in, bottom=1in, left=1in, right=1in]{geometry}  
\usepackage{amsmath, amsthm, amssymb, amsfonts, bbm, bm}  
\usepackage{mathtools}   

\usepackage{enumerate,verbatim}
\usepackage{enumitem}   
\usepackage[authoryear]{natbib}  
\usepackage{url} 
\usepackage{hyperref}
\hypersetup{
    colorlinks=true,
    citecolor=blue,
    linkcolor=blue,
    filecolor=magenta,
    urlcolor=cyan,
}

\usepackage{graphicx} 
\usepackage{arydshln}  
\usepackage[table,xcdraw,dvipsnames]{xcolor}   
\usepackage{hyperref}
\newcommand\myshade{85}
\colorlet{mylinkcolor}{YellowOrange}
\colorlet{mycitecolor}{Aquamarine}
\colorlet{myurlcolor}{violet}

\hypersetup{
  linkcolor  = mylinkcolor!\myshade!black,
  citecolor  = mycitecolor!\myshade!black,
  urlcolor   = myurlcolor!\myshade!black,
  colorlinks = true,
}

\usepackage{pdflscape}
\usepackage{afterpage}
\usepackage{caption}
\usepackage{subcaption}
\usepackage{adjustbox}

\usepackage{diagbox}

\usepackage{multirow}   
\usepackage{tabularx} 

\usepackage{float}   
\usepackage{nameref}   

\usepackage[font=normalsize]{caption}  
\usepackage{subcaption}  
\usepackage[titletoc,title]{appendix}  
\usepackage{longtable} 

\usepackage{kpfonts}
\usepackage[T1]{fontenc}

\usepackage{setspace}
\usepackage{sectsty}
\allsectionsfont{\singlespacing}

\usepackage[compact]{titlesec}
\titlespacing{\section}{0pt}{*0}{*0}
\titlespacing{\subsection}{0pt}{*0}{*0}
\titlespacing{\subsubsection}{0pt}{*0}{*0}

\usepackage{caption}
\usepackage{etoolbox}
\BeforeBeginEnvironment{equation*}{\begin{singlespace}\vspace*{-\baselineskip}}
	\AfterEndEnvironment{equation*}{\end{singlespace}\noindent\ignorespaces}
\BeforeBeginEnvironment{equation}{\begin{singlespace}\vspace*{-\baselineskip}}
	\AfterEndEnvironment{equation}{\end{singlespace}\noindent\ignorespaces}
\BeforeBeginEnvironment{align}{\begin{singlespace}\vspace*{-\baselineskip}}
	\AfterEndEnvironment{align}{\end{singlespace}\noindent\ignorespaces}
\BeforeBeginEnvironment{align*}{\begin{singlespace}\vspace*{-\baselineskip}}
	\AfterEndEnvironment{align*}{\end{singlespace}\noindent\ignorespaces}
\BeforeBeginEnvironment{eqnarray}{\begin{singlespace}\vspace*{-\baselineskip}}
	\AfterEndEnvironment{eqnarray}{\end{singlespace}\noindent\ignorespaces}
\BeforeBeginEnvironment{eqnarray*}{\begin{singlespace}\vspace*{-\baselineskip}}
	\AfterEndEnvironment{eqnarray*}{\end{singlespace}\noindent\ignorespaces}

\usepackage{authblk} 

\graphicspath{ {figs/} }
\input{ccMacro.tex}

\newcommand{\bigO}[1]{ \mathcal{O} \left( #1 \right) }
\newcommand{\mybibsty}{chicago}
\newcommand{\mybibfile}{matvari}

\begin{document}
\renewcommand{\baselinestretch}{1.2}

\if0\blind
{
	\title{\bf Statistical Inference for High-Dimensional \\ Matrix-Variate Factor Models}
	\author[1]{Elynn Y. Chen \thanks{Elynn Chen is supported in part by NSF Grants DMS-1803241.}}
	\author[2]{Jianqing Fan \thanks{Jianqing Fan is Frederick L. Moore ’18 Professor of Finance, Department of Operations Research and Financial Engineering, Princeton University, Princeton, NJ 08544, USA. His research is supported in part by ONR grant N00014-19-1-2120, NSF grant DMS-1712591 and NIH grant 5R01-GM072611-16. Corresponding author: jqfan@princeton.edu.}}
	\affil[1]{EECS, University of California, Berkeley}
	\affil[1,2]{ORFE, Princeton University}
	\date{\today}
	\maketitle
} \fi

\if1\blind
{
    \title{\bf Statistical Inference for High-Dimensional \\Matrix-Variate Factor Models}
    \author{}
	\date{\today}
	\maketitle
	\vspace{9em}
} \fi

\def \defeq{\triangleq}

\def\r#1{\textcolor{red}{#1}}
\def\b#1{\textcolor{blue}{#1}}

\maketitle

\begin{abstract}

This paper considers the estimation and inference of the low-rank components in high-dimensional matrix-variate factor models, where each dimension of the matrix-variates ($p \times q$) is comparable to or greater than the number of observations ($T$).
We propose an estimation method called $\alpha$-PCA that preserves the matrix structure and aggregates mean and contemporary covariance through a hyper-parameter $\alpha$.
We develop an inferential theory, establishing consistency, the rate of convergence, and the limiting distributions, under general conditions that allow for correlations across time, rows, or columns of the noise.
We show both theoretical and empirical methods of choosing the best $\alpha$, depending on the use-case criteria.
Simulation results demonstrate the adequacy of the asymptotic results in approximating the finite sample properties.
The $\alpha$-PCA compares favorably with the existing ones.
Finally, we illustrate its applications with a real numeric data set and two real image data sets.
In all applications, the proposed estimation procedure outperforms previous methods in the power of variance explanation using out-of-sample 10-fold cross-validation.

\vspace{1em}

\end{abstract}


\input{all.tex}

\section*{Acknowledgments}
We would like to thank the Associate Editor and three anonymous reviewers for their comments that helped us improve this paper.
We also thank Ellen Li for preparing the initial codes for the simulation and real data analysis.

\singlespacing
\bibliographystyle{\mybibsty}
\bibliography{\mybibfile}

%

\end{document}

%% file: ccMacro.tex
\newcommand{\bfm}[1]{\ensuremath{\mathbf{#1}}}

\def\bzero{\bfm 0}

   \def\bA{\bfm A}  
     
   \def\bC{\bfm C}  
   \def\bD{\bfm D}  
\def\be{\bfm e}   \def\bE{\bfm E}  \def\EE{\mathbb{E}}
\def\bff{\bfm f}  \def\bF{\bfm F}  
     
   \def\bH{\bfm H}  
   \def\bI{\bfm I}

   \def\bM{\bfm M}

   \def\bQ{\bfm Q}  
   \def\bR{\bfm R}  \def\RR{\mathbb{R}}
   \def\bS{\bfm S}  
     
\def\bu{\bfm u}   \def\bU{\bfm U}  
   \def\bV{\bfm V}  
   \def\bW{\bfm W}  
\def\bx{\bfm x}   \def\bX{\bfm X}  
   \def\bY{\bfm Y}

\def\calD{{\cal  D}} 
\def\calE{{\cal  E}} 
\def\calF{{\cal  F}}

\def\calM{{\cal  M}} 
\def\calN{{\cal  N}} 
\def\calO{{\cal  O}}

\def\calU{{\cal  U}}

\def\calY{{\cal  Y}}

\newcommand{\bfsym}[1]{\ensuremath{\boldsymbol{#1}}}

 \def\bmu{\bfsym {\mu}}

           \def\bepsilon{\bfsym \varepsilon}
              \def\bSigma{\bfsym \Sigma}
         \def\bLambda {\bfsym {\Lambda}}
           \def\bOmega {\bfsym {\Omega}}

 \def\bPhi{\bfsym {\Phi}}
  \def\bPsi{\bfsym {\Psi}}

  \renewcommand{\hat}{\widehat}
  \renewcommand{\tilde}{\widetilde}

            \def\hbC{\hat \bC}

                \def\hbR{\hat \bR}

\providecommand{\abs}[1]{\left\lvert#1\right\rvert}
\providecommand{\norm}[1]{\left\lVert#1\right\rVert}

  \providecommand{\paran}[1]{\left( #1 \right)}
  
  \providecommand{\braces}[1]{\left\{ #1 \right\}}

\newtheorem{assumption}{Assumption}

\newtheorem{theorem}{Theorem}

\newtheorem{remarkx}{Remark}
\newenvironment{remark}{\begin{remarkx} \em}{\end{remarkx}}

\newcommand{\E}[1]{{\mathbb{E}} \left[ #1 \right]}    
\newcommand{\Var}[1]{{\rm Var} \left[ #1 \right]}
\newcommand{\Cov}[1]{{\rm Cov}  \left[ #1 \right]}

\newcommand{\vect}[1]{{\textsc{vec}} \left( #1 \right)}

\newcommand{\Tr}[1]{{\rm Tr} \left( #1 \right) }
\newcommand{\Op}[1]{{\calO_p} \left( #1 \right) }
\newcommand{\op}[1]{{ {\rm o}_p} \left( #1 \right) }
\newcommand{\Oc}[1]{{\rm O} \left( #1 \right) }

\newcounter{question}

\numberwithin{equation}{section}  

\newcounter{CondCounter}

\usepackage[table,xcdraw,dvipsnames]{xcolor}   
\definecolor{royalpurple}{rgb}{0.47, 0.32, 0.66}

\newcommand{\convdist}{\stackrel{\calD}{\longrightarrow}}
\newcommand{\defeq}{\triangleq}

\renewcommand{\bar}{\overline}

%% file: all.tex


\section{Introduction} \label{sec:intro}

%
Large scale matrix-variate data have been widely observed nowadays in diverse fields, such as neuroscience, health care, economics, and social networking.
For example, the monthly import-export volumes among countries naturally form a dynamic sequence of matrix-variates, each of which representing a weighted directional transportation network.
Another example is dynamic panels, such as typical electronic health records (EHRs). In the data-rich intensive care unit (ICU) environment, vitals and other medical tests are measured for different patients at sequential time points. At each time point, the observation is a matrix whose rows represent different patients and whose columns represent demographic information, vitals, lab values, etc.
Thirdly, 2-D image data can also be modeled as matrix-variate data to preserve the spatial information, where each entry of an image matrix corresponds to the intensity of colors of each pixel.
Development of statistical methods for analyzing large scale matrix-variate data is still in its infancy, and as a result, scientists frequently analyze matrix-variate observations by separately modeling each dimension or ‘flattening’ them into vectors. This destroys the intrinsic multi-dimensional structure and misses important patterns in such large scale data with complex structures, and thus leads to sub-optimal results.

The very first questions to ask when facing large scale data with complex structures are: ``Is there a simpler structure behind the massive data set?'' and ``How can we infer the simpler structure from the noisy observations?''
Simpler structures provide better understanding of the problem, reveal more insights into the data, and simplify down-stream analysis.
This paper addresses those questions and provides statistically sound solutions from the perspective of latent factor models.
The proposed method deals with matrix-variate observations directly and works for both independent and weakly-dependent observations.
To the best of our knowledge, we are the first to provide the asymptotic distributions of the estimators for the proposed model.

We specifically consider the following matrix-variate factor model for observations $\bY_t \in \mathbb{R}^{p \times q}$, $1 \le t \le T$:
\begin{equation} \label{eqn:mfm}
\bY_t = \bR \bF_t \bC^\top + \bE_t,
\end{equation}
where $\bY_t$ is driven by a latent factor matrix $\bF_t \in \mathbb{R}^{k \times r}$ of smaller dimensions (i.e. $k \ll p$ and $r \ll q$), plus a noise matrix $\bE_t$.
Matrices $\bR$ and $\bC$ are a $p \times k$ and $q \times r$ row and column loading matrices, respectively.
The noise term $\bE_t$ is assumed to be uncorrelated with $\bF_t$, but is allowed to be weakly correlated across rows, columns and observations.

We propose an estimation procedure, namely $\alpha$-PCA, that aggregates the information in both first and second moments and extract it via a spectral method.
Specifically, we define the following statistics
\begin{eqnarray}
& \hat \bM_R & \defeq \frac{1}{pq}\paran{ \paran{1+\alpha}\cdot\bar\bY\bar \bY^\top + \frac{1}{T}\sum_{t=1}^{T}\paran{\bY_t-\bar\bY}\paran{\bY_t-\bar\bY}^\top},  \label{eqn:hat-MR-aggr} \\
& \hat \bM_C & \defeq \frac{1}{pq} \paran{\paran{1+\alpha}\cdot\bar\bY^\top\bar\bY + \frac{1}{T}\sum_{t=1}^{T}\paran{\bY_t - \bar \bY}^\top \paran{\bY_t - \bar \bY}}, \label{eqn:hat-MC-aggr}
\end{eqnarray}
where $\alpha\in[-1,+\infty)$ is a hyper-parameter balancing the information of the first and second moments,  $\bar\bY=\frac{1}{T}\sum_{t=1}^{T}\bY_t$ is the sample mean,  $\frac{1}{T}\sum_{t=1}^{T}\paran{\bY_t-\bar\bY}\paran{\bY_t-\bar\bY}^\top$ and  $\frac{1}{T}\sum_{t=1}^{T}\paran{\bY_t - \bar \bY}^\top \paran{\bY_t - \bar \bY}$ are the sample \textit{row} and \textit{column covariance matrix}, respectively.
Estimations of $\bR$ and $\bC$ can be obtained, respectively, as $\sqrt{p}$ times the top $k$ eigenvectors of $\hat\bM_R$ and $\sqrt{q}$ times the top $r$ eigenvectors of $\hat\bM_C$, in descending order by corresponding eigenvalues.
We explain its interpretation and relations to several estimation procedures in Section \ref{sec:estimation}.

In the community of image signal processing, model \eqref{eqn:mfm} and estimation methods such as $(2D)^2$-PCA have been actively studied \citep{yang2004two, zhang20052d,kong2005generalized,pang2008binary,kwak2008principal,li2010l1,meng2012improve,wang2015robust}.
However, their studies mainly focus on the algorithmic properties and give no statistical properties on the estimators that are highly demanded in the medical, economics, and social applications nowadays.
The proposed $\alpha$-PCA aggregates the first moment (weighted by $1+\alpha$) and the second moment, where $\alpha\in[-1,+\infty)$ is a hyper-parameter in \eqref{eqn:hat-MR-aggr} and \eqref{eqn:hat-MC-aggr}.
It encompasses $(2D)^2$-PCA as a special case of $\alpha=-1$, which is not a best choice in general.
We show theoretically and empirically how to choose optimal $\alpha$ under different criteria, such as achieving most efficient estimators and providing most accurate predictors.
Also, we are the first to apply model \eqref{eqn:mfm} to provide convergence and asymptotic normality results of the estimators under a very general setting.

With respect to statistical analyses, \cite{wang2019factor} and \cite{chen2019constrained} consider a similar model in the bilinear form \eqref{eqn:mfm}, yet under a very different setting where $\bE_t$ is assumed to be white noise \citep{ lam2012factor, lam2011estimation}.
\cite{chen2019factor} extends previous results to time series of tensor observations, again assuming noise tensors are not temporally correlated.
This line of research discards contemporaneous covariance and utilizes only the auto-covariance between $\bY_t$ and $\bY_{t-h}$ with $h \ge 1$.
The white noise assumption for $\bE_t$ simplifies the problem by removing the error covariance $\E{\bE_t \bE_{t-h}^\top} = 0$ ($h\ge1$) from $\E{(\bY_t - \EE \bY_t) (\bY_{t-h} - \EE \bY_{t-h})^\top}$, but the resulting data can have little information for the quantity that we would like to learn.
Indeed, the most informative component $\E{(\bY_t - \EE \bY_t) (\bY_{t} - \EE \bY_{t})^\top}$ is excluded.
The $\ell_2$ convergence rates obtained by \cite{wang2019factor} for the estimators of $\bR$ and $\bC$ are both $T^{-1/2}$ with  strong factors (i.e. Assumption \ref{assume:loading_matrix} in Section \ref{sec:theory}).
Although they use auto-covariance matrices, their results are comparable to the noiseless version of our model \eqref{eqn:mfm}.
Under the noiseless setting when term $\bE_t$ in equation \eqref{eqn:mfm} is ignored, our results give faster convergence rates of ${(qT)}^{-1/2}$ for $\bR$ and ${(pT)}^{-1/2}$ for $\bC$ with strong factors, same as those obtained in \cite{chen2019factor} for order-2 tensor observations.

Even in the case of $\alpha = -1$, our models and methods are very different.
We need to deal with the bias term $\E{\bE_t \bE_t^\top} \ne 0$, while the analyses in \cite{wang2019factor,chen2019factor} are largely simplified by assuming $\bE_t$ as white noise and not including contemporaneous covariance $\E{(\bY_t - \EE \bY_t) (\bY_{t} - \EE \bY_{t})^\top}$.
Furthermore, our assumption is more general in that $\bE_t$ is allowed to be weakly correlated across rows, columns and observations.  

The contributions of this paper are three folds.
Firstly, we expand considerably the scope of applicability of \cite{wang2019factor} and related work, making the theory and methods useful for a wider range of applications.
The previous work uses only cross-covariance to learn the latent factors and factor loadings.
This not only requires the restrictive assumption that $\{\bE_t\}$ is a white noise series, but also becomes ineffective when auto-correlations are weak.
This makes the procedure inapplicable to i.i.d matrix-variate data, such as gene or proteomic expression data across samples and multiple image data illustrated in Section \ref{sec:application}.2.
It can not be applied to financial return data due to the efficient market hypothesis.
In contrast, we use the most informative piece of information: the contemporary covariance matrix.
This modification makes the procedure applicable to i.i.d matrix-variate data and weakly auto-correlated data.

In addition, we point out that the first moments also provide useful information and thoroughly incorporate this by aggregating it with the second moments via a weighted spectral method.
Theorem 2 shows precisely how much the benefit is (if any).
We show how to choose the parameter $\alpha$ in real applications and further point out a generalization of this idea to yield an even more powerful method by incorporating the auto-covariance as well \citep{wang2019factor}.

On theoretical aspects, we establish new results on the asymptotic normality and the optimal $\alpha$ of the $\alpha$-PCA.
They are useful in constructing the confidence intervals of the estimators and also in choosing the optimal values of $\alpha$.

\subsection{More related work}

Besides the literature in image processing and matrix-variate factor models, this paper is related to the literature of vector factor models and statistical tensor data analysis.
Model \eqref{eqn:mfm} can be seen as a generalization of the vector factor model \citep{bai2002determining, bai2003inferential,fan2013large,chang2015high, fan2018robust,fan2020statistical} to matrix-variate data.
Solving model \eqref{eqn:mfm} directly achieves a better convergence rate in a high-dimensional regime than that which results from applying the vector factor model to vectorized observations.
In particular, consider the following vectorized version of model \eqref{eqn:mfm}:
\begin{equation} \label{eqn:mfm-vec}
\vect{\bY_t} = ( \bC \otimes \bR ) \cdot \vect{\bF_t} + \vect{\bE_t},
\end{equation}
where $\vect{\bY_t} \in \mathbb{R}^{pq}$ and $\vect{\bF_t} \in \mathbb{R}^{kr} $.
The $\ell_2$ convergence rate for $\hat{ \bC \otimes \bR }$ obtained by traditional PCA \citep{bai2003inferential, bai2002determining} is $\min\{pq, T\}^{-1/2}$, without adopting the tensor structure in the loading matrix.
Under similar assumptions, solving model \eqref{eqn:mfm} directly gives a $\ell_2$ convergence rate of $\min\{p, Tq\}^{-1/2}$ for $\hat \bR$ and $\min\{q, Tp\}^{-1/2}$ for $\hat \bC$.
In a high-dimensional regime where $p, q > T$, our method gives better results.
Furthermore, we obtain $\hat \bR$ and $\hat \bC$ by directly solving model~\eqref{eqn:mfm}, more specifically applying PCA to \eqref{eqn:hat-MR-aggr} and \eqref{eqn:hat-MC-aggr}, while one needs to carry out a second step to estimate $\hat \bR$ and $\hat \bC$ from $\hat{ \bC \otimes \bR }$, which may incur further errors \citep{cai2019kopa}.
See remarks after Theorem \ref{thm:Rhat_F_norm_convg} and \ref{thm:asymp normality of R/C}  for more discussion.

Tensor decomposition \citep{kolda2009tensor, kolda2006multilinear} has also been applied to matrix-variate observations. Note that $\{\bY_t\}_{1 \le t \le T}$ form an order-3 tensor of dimension $p \times q \times T$  by stacking $\bY_t$ along the third mode ${1 \le t \le T}$.
Statistical convergence rates in Frobenius norm have been studied in \cite{zhang2018tensor} under the assumption of homogeneous entries in tensor.
However, vanilla Tucker decomposition does not apply directly here. (See Remark 3 for more discussion.)
We allow correlations across rows, columns and observations in $\bE_t$ and also derived the asymptotic normalities for $\hat \bR$ and $\hat \bC$.
Additionally, by focusing on the simplest multi-dimensional objects and connecting them with the matrix-variate normal distribution, our analysis provides statistical insights that are potentially helpful in understanding the behavior of higher-order multi-dimensional observations.
Generalizing our method to higher-order tensor decomposition is an interesting direction for future research.

\subsection{Notation and organization}

We use lowercase letter $x$, boldface letter $\bx$, and boldface capital letter $\bX$ to denote scalar, vector and matrix, respectively.
We use $\bX_{i\cdot}$, $\bX_{\cdot j}$, and $x_{ij}$ to denote the $i$-th row, $j$-th column, and $(i,j)$-th element of a matrix $\bX$, respectively.
For a matrix $\bX$, we use the following matrix norms: maximum norm $\norm{\bX}_{max} \; \defeq \; \underset{ij}{\max} \; \lvert x_{ij} \rvert$, $\ell_1$-norm $\norm{\bX}_1 \; \defeq \; \underset{j}{\max} \; \sum_{i} \lvert x_{ij} \rvert$, $\ell_{\infty}$-norm $\norm{\bX}_{\infty} \; \defeq \; \underset{i}{\max} \; \sum_{j} \lvert x_{ij} \rvert$, and $\ell_2$-norm $\norm{\bX}\; \defeq \; \sigma_1$, where $\sigma_1$ is the largest singular value $\{\sigma_i\}$ of $\bX$ with $\sigma_i$ being the $i$-th largest square root of eigenvalues of $\bX^\top \bX$.
We also use $\norm{\bX}$ for $\ell_2$ norm.
When $\bX$ is a square matrix, we denote by $\Tr{\bX}$, $\lambda_{max} \paran{\bX}$, and $\lambda_{min} \paran{\bX}$ the trace, maximum and minimum singular value of $\bX$, respectively.
We let $[n] \defeq \braces{1, \ldots, n}$ denote the set of integers from $1$ to $n$.

The rest of this paper is organized as follows.
In Section \ref{sec:estimation}, we introduce estimation method for model \eqref{eqn:mfm}.
We develop the asymptotic normality for the estimated loading matrices in Section \ref{sec:theory} and provide consistent estimators of the asymptotic variance-covariance matrices in Section \ref{sec:cov-est}.
In Section \ref{sec:simulation}, we study the finite sample performance of our estimation via simulation.
Section \ref{sec:application} provides empirical studies.
Section \ref{sec:conclusion} concludes.
All proofs and technique lemmas are relegated to Appendix A and B in the supplemental materials.

\section{Estimation}  \label{sec:estimation}


\subsection{Model identification}
We only observe $\bY_t$ and everything on the right hand side of model \eqref{eqn:mfm} is unknown.
Separation of the signal part $\bS_t = \bR \bF_t \bC^\top$ and noise part $\bE_t$ can be achieved by the pervasiveness assumption (i.e. Assumption \ref{assume:loading_matrix}) on loading matrices $\bR$ and $\bC$ and the bounded eigenvalues assumption (i.e. Assumption \ref{assume:cross_row_col_corr}) of noise row and column covariances in Section \ref{sec:theory}.
The latent factor matrix $\bF_t$ and loading matrices $\bR$ and $\bC$ are not separately identifiable.
However, they can be estimated up to an invertible matrix transformation.
Particularly, let $\bH_R \in \RR^{k \times k}$ and $\bH_C \in \RR^{r \times r}$ be two non-singular matrices. The triplets $(\bR, \bF_t, \bC)$ and $(\bR \bH_R^{-1}, \; \bH_R \bF_t \bH_C^\top, \; \bC \bH_C^{-1})$ are equivalent under model \eqref{eqn:mfm}.

Thus instead of the ground truth $\bR^\star$, $\bF_t^\star$ and $\bC^\star$, we aim at estimating transformations of the true values. Without loss of generality, restrict our estimator $\hat \bR$ and $\hat \bC$ such that
\begin{equation} \label{eqn:orthonormal-cond'n}
\frac{1}{p} \hat \bR^\top \hat \bR = \bI, \quad \text{and} \quad \frac{1}{q} \hat \bC^\top \hat \bC = \bI.
\end{equation}
As shown in the Theorem \ref{thm:asymp normality of R/C}, for any ground truth $\bR^\star$, $\bC^\star$, $\bF_t^\star$ and our estimator $\hat\bR$ ($\hat\bC$), there exists an \textit{invertible} $\bH_R$ ($\bH_C$) given in \eqref{eqn:def_HR}  (\eqref{eqn:def_HC}) such that $\hat\bR$ ($\hat\bC$) is a close estimator of $\bR^\star\bH_R$ ($\bC^\star\bH_C$) and $\hat\bF_t$ is an estimator of $ \bH_R^{-1}\bF_t^\star{\bH_C^{-1}}^\top$.
Knowing $\bR^\star\bH_R$, $\bC^\star\bH_C$, and $\bH_R^{-1}\bF_t^\star{\bH_C^{-1}}^\top$ is as good as knowing true $\bR^\star$, $\bC^\star$ and $\bF_t^\star$ for many purposes.
For example, in regression analysis or time series prediction, using $ \bH_R^{-1}\bF_t^\star{\bH_C^{-1}}^\top$ as the regressor will give the same predicted value as using $\bF_t^\star$ as the regressor.
Note that the \textit{true} $\bR^\star$ and $\bC^\star$ do not necessarily satisfy (\ref{eqn:orthonormal-cond'n}).
If they do, then $\bH_R$ and $\bH_C$ approach orthogonal matrices in the limit.

\subsection{Estimation based on spectral aggregation} \label{sec:aggr}

Note that the first moment $\E{\bY_t} = \bR \E{\bF_t} \bC^\top $, which contains also the information of unknown parameters. Similarly, the second moment
\begin{equation*}
   \E {\paran{\bY_t-\E {\bY_t}} \paran{\bY_t-\E{\bY_t}}^\top } = \bR \E{(\bF_t-\E{\bF_t}) \bC^\top\bC (\bF_t-\E{\bF_t})^\top}\bR^\top + \E {\bE_t \bE_t^\top}
\end{equation*}
also contains information about the unknown parameters.  In particular, after noticing the matrix $\E{(\bF_t-\E{\bF_t}) \bC^\top\bC (\bF_t-\E{\bF_t})^\top}$ is of rank $k $ under some mild conditions and ignoring the second term (as justified by the pervasive assumption below), it is easy to see $\bR$ is the same as the top $k$ eigenvectors of the second moment, up to an affine transformation.  This justifies our spectral method based on \eqref{eqn:hat-MR-aggr} and \eqref{eqn:hat-MC-aggr} introduced in the introduction.

Let $\tilde\alpha=\sqrt{\alpha +1}-1$ and
\begin{equation*} \label{eqn:tildeY}
\tilde\bY_t \defeq \bY_t + \tilde\alpha\bar\bY, \quad \tilde\bF_t \defeq \bF_t + \tilde\alpha\bar\bF_t, \quad\text{and}\quad \tilde\bE_t \defeq \bE_t + \tilde\alpha\bar\bE_t.
\end{equation*}
Then we have
\begin{equation} \label{eqn:proj-Y}
\tilde\bY_t = \bR\tilde\bF_t\bC^\top + \tilde\bE_t.
\end{equation}
Equations \eqref{eqn:hat-MR-aggr} and \eqref{eqn:hat-MC-aggr} can be equivalently written as
\begin{equation} \label{eqn:M-tildeY}
\hat\bM_R = \frac{1}{pqT} \sum_{t=1}^{T} \tilde\bY_t\tilde\bY_t^\top, \quad\text{and}\quad\hat\bM_C = \frac{1}{pqT} \sum_{t=1}^{T}\tilde\bY_t^\top\tilde\bY_t,
\end{equation}
which can be viewed as the statistics defined on the transformed data $\tilde\bY_t$. The special case for $\alpha=-1$ corresponds to the sample row and column covariance matrices of the original data.

The estimators $\hat \bR$ and $\hat \bC$ are respectively obtained as $\sqrt{p}$ times the top $k$ eigenvectors of $\hat\bM_R$ and $\sqrt{q}$ times the top $r$ eigenvectors of $\hat\bM_C$, in descending order by corresponding eigenvalues.

\begin{remark}\label{remark:est-1}
    {\em Auto-covariance based estimation.}
    \cite{wang2019factor} and \cite{chen2019constrained} consider a similar model in the bilinear form \eqref{eqn:mfm}, yet under a very different setting where $\bE_t$ is assumed to be white noise.
    The major methodological difference is that \cite{wang2019factor} utilizes only the auto-covariance between $\bY_t$ and $\bY_{t-h}$ with $h \ge 1$, discarding the covariance of $\bY_t$ totally.
    When the data is temporally independent or weakly correlated, the population auto-covariance of lag $h\ge 1$ (signal) is equal to or close to zero and the sample auto-covariance has very low signal noise ratio.  In other words, this kind of methods can not be applied to the cross-sectional data such as high-throughput genomics measurements where $t$ indices an individual or financial return data where predicability is low due to efficient markets.
    The performance comparisons in Section 6 also confirm this concern in real data sets.
\end{remark}

\begin{remark} \label{remark:est-2}
    {\em Spectral aggregation.}
    The proposed method falls in the category of spectral methods which are based on eigen-decomposition or singular value decomposition of moments-type statistics, i.e. matrices $\bM_R$ and $\bM_C$.
    One major difference between statistical methods in this category is how the statistics $\bM_R$ ($\bM_C$) is constructed.
    \cite{wang2019factor} and \cite{chen2019constrained} construct $\bM$ using the auto-covariance and derive the properties of their auto-covariance-based estimators under the assumption that $\bE_t$ is white noise.
    They require that the factors be pervasive cross-section $(p,q)$, and also that the factors be temporally dependent (otherwise the signal part equals zero.)
    The present paper constructs $\bM_R$ ($\bM_C$) using covariance and the theoretical properties are derived under a different set assumptions.

    A very interesting point raised by the referee is that whether we can use both covariance and auto-covariance for spectral aggregation.
    \cite{forni2015dynamic,forni2017dynamic} proposed a full dynamic factor model for vector time series which include both covariance and auto-covariance.
    While we are considering a static factor model \citep{bai2003inferential} here, the information of first moment, covariance and lag-$h$ auto-covariance for $h \ge 1$ can be aggregated to yield an even better performance, as long as $\bE_t$ is white noise.  See \cite{fan2019optimal} for the methods and the results on spectral aggregations.
\end{remark}

\begin{remark} \label{remark:est-3}
    {\em Tensor decomposition.}
    Matrix-variate time series $\bY_t$, $t\in[T]$, is the 2nd-order tensor time series.
    Also, it can be stacked along a third mode of time to form a 3rd-order tensor $\calY \in \RR^{p\times q\times T}$.
    Tucker decomposition \citep{kolda2009tensor, kolda2006multilinear} can be applied to the 3rd-order tensor $\calY$ directly.
    Model \eqref{eqn:mfm} can be written equivalently as a noisy Tucker decomposition $\calY = \calF \times_1 \bR \times_2 \bC \times_3 \bI_{T} + \calE$ where $\times_m$ is the mode $m$ tensor product and $\bI_{T}$ is the identity matrix of size $T$.
    At the same time, Tucker decomposition can be applied to the covariance tensor defined as $\Cov{\bY_t} = \Cov{\bF_t} \times_1 \bR \times_2 \bR \times_3 \bC \times_4 \bC$, where $\Cov{\bY_t}\in\RR^{p\times p \times q \times q}$ with the $ijkl$-th element being $\Cov{ y_{t,ik} y_{t,jl} }$.
    These two problems are constrained Tucker decomposition: the formal restricts that the time-mode loading matrix is the identity matrix $\bI_T$, while the latter restricts that two loadings are exactly the same.
    It is of great interest to extend the current algorithms and theories on Tucker decomposition (See \cite{zhang2018tensor} and references therein) to such constrained Tucker decomposition problems.
\end{remark}

\subsection{Relations to LS, ML and PCA estimators} \label{sec:2mm}

In this section, we provide more interpretation of $\alpha$-PCA.
Our estimation approximates the least squares and maximum likelihood estimators and encompasses PCA type of estimators as a special case with $\alpha=-1$.
The proposed estimator in Section \ref{sec:aggr} approximately minimizes jointly the unexplained variation and bias:
\begin{equation} \label{eqn:opt-obj-l2}
\begin{aligned}
& \underset{\bR, \bC, \{\bF_t\}_{t=1}^T }{\text{minimize}} & & \paran{1+\alpha}\underbrace{\frac{1}{pq}\norm{\bar\bY - \bR\bar\bF\bC^\top}_F^2}_{\text{sample bias}} + \underbrace{\frac{1}{pqT} \sum_{t=1}^{T} \norm{\bY_t-\bR\bF_t\bC^\top}_F^2}_{\text{sample variance}}\\
& \text{subject to} & & \frac{1}{p} \bR^\top \bR = \bI, \enspace \frac{1}{q} \bC^\top \bC = \bI.
\end{aligned}
\end{equation}
The special case $\alpha=-1$ corresponds to the \textit{least squares estimator}.
However, \eqref{eqn:opt-obj-l2} is non-convex.
Thus, instead of solving \eqref{eqn:opt-obj-l2} directly, we may consider an approximate solutions by maximizing row and column variances respectively after projection.

Firstly, $\braces{\bY_t}_{t\in[T]}$ are projected onto $\bR$ and maximize the row variances of $\bR^\top \bY_t$ under the constraint that $\frac{1}{p} \bR^\top \bR = \bI$. On the population level, that is,
\begin{equation*} \label{eqn:opt-obj-R}
\begin{aligned}
& \underset{\bR}{\text{maximize}} & & \Tr{\E{\paran{1+\alpha}\paran{\bR^\top \bar\bY}\paran{\bR^\top\bar\bY}^\top + \paran{\bR^\top \bY_t-\EE[\bR^\top\bY_t]}\paran{\bR^\top \bY_t-\EE[\bR^\top \bY_t]}^\top}} \\
& & & =\;\Tr{pq\cdot\bR^\top\bM_R \bR}, \\
& \text{subject to} & & \frac{1}{p} \bR^\top \bR = \bI,
\end{aligned}
\end{equation*}
where
\begin{equation*}
\bM_R \defeq \paran{1+\alpha}\bM_R^{(1)}+\bM_R^{(2)},\quad \bM_R^{(1)} \defeq \frac{1}{pq}\;\E{\bar\bY\,\bar\bY^\top},\quad\text{and}\quad\bM_R^{(2)} \defeq \frac{1}{pq}\;\E{\paran{\bY_t-\EE[\bar\bY]}\paran{\bY_t-\EE[\bar\bY]}^\top}.
\end{equation*}
Similar expressions can be obtained by using the projections onto $\bC$ and maximize the column variances of $\bY_t \bC$.
Note that a factor of $\frac{1}{pq}$ does not change the column space of $\bM_R$ or $\bM_C$, but will facilitate theoretical analysis and stabilize numerical computation as $p$ and $q$ increase.

With $T$ observations $\braces{ \bY_t }_{t \in [T]}$,  we replace the population mean and covariance matrix by their sample versions
and obtain the maximizer $\hat\bR$ ($\hat\bC$) comprised of $\sqrt{p}$ ($\sqrt{q}$) times top $k$ ($r$) eigenvectors of $\hat\bM_R$ ($\hat\bM_C$) in descending order by corresponding eigenvalues. Thus the estimator defined in Section \ref{sec:aggr} approximately solves \eqref{eqn:opt-obj-l2}.

\subsection{Estimation of the factor and signal matrices}

After estimating $\hat\bR$ and $\hat\bC$ by spectral aggregation described in Section \ref{sec:aggr}, we obtain an estimator of $\bF_t$ using condition (\ref{eqn:orthonormal-cond'n}):
\begin{equation} \label{eqn:F_t_est}
\hat \bF_t = \frac{1}{p q} \hat \bR^\top \bY_t \hat \bC.
\end{equation}
The signal part $\bS_t = \bR \bF_t \bC^\top$ can be estimated by
\begin{equation} \label{eqn:S_t_est}
\hat \bS_t = \frac{1}{p q} \hat \bR \hat \bR^\top \bY_t \hat \bC \hat \bC^\top.
\end{equation}

The above estimation procedure assumes that the latent dimensions $k \times r$ are known.
However, in practice we need to estimate $k$ and $r$ as well.
To determine $k$ and $r$ we could use: (a) the eigenvalue ratio-based estimator, proposed by \cite{ahn2013eigenvalue};
(b) the Scree plot which is standard in principal component analysis.
Let $\hat{\lambda}_1 \ge \hat{\lambda}\ge \cdots \ge \hat{\lambda}_{k} \ge 0$ be the ordered eigenvalues of $\hat \bM_R$. The ratio-based estimator for $k$ is defined as
\begin{equation}  \label{eqn:eigen-ratio}
\hat{k} = \underset{{1 \le j \le k_{\max}}}{\arg \max} \frac{\hat{\lambda}_{j}}{\hat{\lambda}_{j+1}},
\end{equation}
where $k_{\max}$ is a given upper bound.
In practice we may take $k_{\max} = \lceil p/2 \rceil$ or $k_{\max} = \lceil p/3 \rceil$ according to \cite{ahn2013eigenvalue}.
Ratio estimator $\hat r$ is defined similarly with respect to $\hat \bM_C$.
Adjustments of estimated eigenvalues are needed when the optimal $k$ grows with $p$ \citep{fan2020estimating}.

In the next section, we establish theoretical results showing that under high dimensional settings, 
$\hat{\bR}$, $\hat{\bC}$ and $\hat\bF_t$ are consistent estimators under known fixed $k$ and $r$. In addition,  we obtain the asymptotic distributions for $\hat{\bR}$ and $\hat{\bC}$.

\section{Theoretical Properties}   \label{sec:theory}

We first state all the necessary assumptions used in the following sections.
To simplify notation, we drop the $\star$ superscript and let $\bF_t \in \RR^{k \times r}$,  $\bR$, and $\bC$ be the true latent factor, row and column loading matrices, respectively.
Let $\bar \bF = \frac{1}{T} \sum_{t=1}^{T} \bF_t$ and $\bar \bE = \frac{1}{T} \sum_{t=1}^{T} \bE_t$ be the sample means of the factors and the noise, respectively.
Entries in the matrices are respectively denoted as $\bar f_{ij}$ and $\bar e_{ij}$.

%

\begin{assumption} \label{assume:alpha-mixing}
    \textbf{$\alpha$-mixing.}
    The vectorized factor $\vect{\bF_t}$ and noise $\vect{\bE_t}$ are $\alpha$-mixing. Specifically, a vector process $\{ \bx_t, t=0,\pm 1, \pm 2, \cdots \}$ is $\alpha$-mixing if, for some $\gamma > 2$, the mixing coefficients satisfy the condition that
    $$\sum_{h=1}^{\infty} \alpha(h)^{1-2/\gamma} < \infty,$$
    where $ \alpha(h)=\underset{\tau}{\sup} \underset{A \in \mathcal{F}_{-\infty}^{\tau}, B \in \mathcal{F}_{\tau+h}^{\infty}}{\sup} \left| P(A \cap B) - P(A)P(B) \right| $ and $\mathcal{F}_{\tau}^s$ is the $\sigma$-field generated by $\{\bx_t: {\tau} \le t \le s \}$.
\end{assumption}

\begin{assumption} \label{assume:factor-noise-matrix}
    \textbf{Factor and noise matrices.}
     There exists a positive constant $C < \infty$ such that for all $N$ and $T$,
    \begin{enumerate}[label=(\alph*)]
        \item Factor matrix $\bF_t$ is of fixed dimension $k \times r$ and $\EE\norm{\bF_t}^4 \le C$.

        \item For all $i \in [p]$, $j \in [q]$ and $t \in [T]$, $\E{e_{t, ij}} = 0$ and $\EE \lvert e_{t, ij} \rvert^8 \le C$.

        \item Factor and noise are uncorrelated, that is, $\E{e_{t,ij} f_{s,lh}} = 0$ for any $t, s \in [T]$, $i\in[p]$, $j\in [q]$, $l\in[k]$, and $h\in[r]$.
    \end{enumerate}
\end{assumption}

\begin{assumption} \label{assume:loading_matrix}
    \textbf{Loading matrix.}
    For each row of $\bR$, $\norm{\bR_{i \cdot}} = \bigO{1}$, and, as $p,q \rightarrow \infty$, we have $\norm{p^{-1} {\bR}^\top \bR - \bOmega_R} \longrightarrow 0$ for some $k \times k$ positive definite matrix $\bOmega_R$.
    For each row of  $\bC$,  $\norm{\bC_{i \cdot}} = \bigO{1}$, and, as $p,q \rightarrow \infty$, $\norm{q^{-1} {\bC}^\top \bC - \bOmega_C} \longrightarrow 0$ for some $r \times r$ positive definite matrix $\bOmega_C$.
\end{assumption}

Assumption \ref{assume:loading_matrix} is an extension of the \textit{pervasive assumption} \citep{stock2002forecasting} to the matrix variate data.
It ensures that each row and column of the factor matrix $\bF_t$ has a nontrivial contribution to the variance of rows and columns of $\bY_t$.
Thus our analysis is in the regime of ``strong factors''
that they lead to exploding eigenvalues relative to the idiosyncratic eigenvalues.

Note that Assumption \ref{assume:alpha-mixing}  only deals with {\em temporal dependence}.
The matrix dimension $p$ and $q$ also determine the convergence rates, which is affected by the cross-row and cross-column dependence.
Thus we need Assumptions \ref{assume:cross_row_col_corr} and \ref{assume:CLT} below so that the information accumulated over rows ($p$) or columns ($q$) is also useful.
Assumption \ref{assume:cross_row_col_corr} holds automatically when the errors $\bE_t$ are i.i.d. over rows and columns for any $t$.

\begin{assumption} \label{assume:cross_row_col_corr}
    \textbf{Cross row (column) correlation of noise $\bE_t$.}
    There exists some positive constant $C < \infty$ such that,
    \begin{enumerate}[label=(\alph*)]

        \item Let  $\bU_E = \E{\frac{1}{qT} \sum_{t=1}^{T}\bE_t \bE_t^T}$ and $\bV_E = \E{\frac{1}{pT} \sum_{t=1}^{T}\bE_t^T \bE_t}$, we assume $\norm{ \bU_E }_1 \le C$ and $\norm{ \bV_E }_1 \le C$.

        \item For all row $i \in [p]$ and column $j \in [q]$ and $t \in [T]$, we assume $
        \sum_{\substack{l\in [p] \\ l\ne i}} \sum_{\substack{h\in [q] \\ h\ne j}} \abs{\E{e_{t,ij} e_{t,lh}}} \le C$.


        \item For any row $i, l \in[p]$, any time $t \in [T]$, and any column $j\in [q]$,
        \begin{equation*}
        \sum_{m\in[p]}
        \sum_{s\in[T]}
        \sum_{\substack{h\in [q] \\ h\ne j}}
        \abs{\Cov{ e_{t, ij} e_{t, lj}, \; e_{s, ih} e_{s, mh} } } \le C
        \end{equation*}
        Similar, for any column $j, h\in[q]$, any time $t \in [T]$, and any row $i \in [p]$,
        \begin{equation*}
            \sum_{m \in [q]}
            \sum_{s \in [T]}
            \sum_{\substack{l \in [p] \\ l \ne i}}
            \abs{\Cov{ e_{t, ij}e_{t, ih}, \; e_{s, lj} e_{s, lm} } } \le C
    \end{equation*}

    \end{enumerate}
\end{assumption}

To better interpret the cross-row/column correlation of noise terms in Assumption \ref{assume:cross_row_col_corr}, we consider the special case when $\bE_t$ follows an i.i.d matrix-variate normal distribution $\calM \calN_{p \times q} \left( \bzero, \tilde \bU_E, \tilde \bV_E \right)$.
Then
\[
\bU_E = \E{\frac{1}{qT} \sum_{t=1}^{T}\bE_t \bE_t^T} = \tilde \bU_E \cdot \frac{1}{q} \, \Tr{\tilde \bV_E}.
\]
Given that $ \frac{1}{q} \, \Tr{\tilde \bV_E}= \Oc{1}$, Assumption \ref{assume:cross_row_col_corr} (a) requires that the row covariance $\tilde \bU_E$ of the noise matrix satisfies $\norm{\tilde \bU_E}_{1} < c$.
Similarly,  we require $\norm{\tilde \bV_E}_{1} < c$.
It is satisfied if $\bU_E$ and $\bV_E$ are diagonal matrices, or more generally sparse matrices.
Given Assumption \ref{assume:factor-noise-matrix}, the remaining assumptions in Assumption \ref{assume:cross_row_col_corr} are satisfied if $e_{t, ij}$ are independent for all $i$, $j$, and $t$.
We allow weak correlations across $i$, $j$ or $t$ in the noise, which is more general than the i.i.d.~assumption in tensor decomposition literature \citep{zhang2018tensor}.

\begin{assumption}  \label{assume:CLT}
    There exists $m > 2$, $1 < a, b < \infty$, $1/a + 1/b = 1$, such that, for some positive $C < \infty$,
    \begin{enumerate}[label=(\alph*)]
        \item For any $l\in [k]$, $i\in[p]$, and $t\in [T]$, $\E{ \abs{\frac{1}{\sqrt{q}} \sum_{j=1}^q e_{t, i j}}^{mb} } = \bigO{1}$, $\E{ \norm{\frac{1}{\sqrt{q}} \sum_{j=1}^q \bC_{j \cdot} e_{t, i j}}^{mb} } = \bigO{1}$, and $\E{\norm{\bff_{t, l \cdot}}^{ma}} \le C$.
        \item For any $h \in [r]$, $j \in[q]$, and $t\in [T]$,
        $\E{ \abs{\frac{1}{\sqrt{p}} \sum_{i=1}^p e_{t, i j}}^{mb} } = \bigO{1}$,
        $\E{ \norm{\frac{1}{\sqrt{p}} \sum_{i=1}^p \bR_{i \cdot} e_{t, i j}}^{mb} } = \bigO{1}$,
        and $\E{\norm{\bff_{t, \cdot h}}^{ma}} \le C$.
        \item For any $t\in [T]$, $\E{\abs{\frac{1}{\sqrt{pq}} \sum_{i=1}^p \sum_{j=1}^q e_{t, i j}}^{mb} } = \bigO{1}$ and $\E{\norm{\frac{1}{\sqrt{pq}} \sum_{i=1}^p \sum_{j=1}^q \bR_{i \cdot}\bC_{j \cdot}^\top e_{t, i j}}^{mb} } = \bigO{1}$.
    \end{enumerate}
\end{assumption}

Assumption \ref{assume:CLT} is satisfied by Gaussian noise $\bE_t$ with i.i.d rows and columns.
Specifically, if $e_{t, ij} \sim \calN\paran{0, \sigma^2}$ are i.i.d.~over $i\in[p]$ and $j\in[q]$, then $\frac{1}{\sqrt{p}} \sum_{i=1}^{p} e_{t, ij} \convdist \calN\paran{0, \sigma^2}$,
$\frac{1}{\sqrt{p}} \sum_{i=1}^{p} \bR_{i \cdot} e_{t, ij} \convdist \calN\paran{ 0, \sigma^2 \bOmega_R }$,
$\frac{1}{\sqrt{pq}} \sum_{i=1}^{p} \sum_{j=1}^{q} e_{t, ij} \convdist \calN\paran{0, \sigma^2}$,
and $\frac{1}{\sqrt{pq}} \sum_{i=1}^p \sum_{j=1}^q (\bC_{j \cdot} \otimes \bR_{i \cdot}) e_{t, i j} \convdist \calN\paran{0, \sigma^2 \bOmega_C \otimes \bOmega_R}$.
Thus, Assumption \ref{assume:CLT} on the noise part is satisfied by choosing $m=2$ and $a = b = 2$.
It is imposed to guarantee the $\sqrt{pT}$ or $\sqrt{qT}$ convergence rate (rather than $\sqrt{T}$) when rows or columns of $\bE_t$ are not independent.
It will not be needed when the errors $\bE_t$ are i.i.d.~over rows and columns for any $t$ and are independent of the factor $\bF_t$,  with assumed moments conditions.
We include them here to allow for \textit{weakly cross-row (-column) and temporal correlations}.

Now, we are ready to present theoretical properties of our estimators. To facilitate the analysis, we first introduce auxiliary matrices $\bH_R$, $\bH_C$, $\bV_{R, pqT}$ and $\bV_{C, pqT}$.
As noted previously, $\bR$, $\bC$ and $\bF_t$ are not separately identifiable.
We show in the following that, for any ground truth $\bR$, $\bC$ and $\bF_t$ and our estimator $\hat \bR$ ($\hat \bC$), there exists an invertible matrix $\bH_R$ ($\bH_C$) such that $\hat \bR$ ($\hat \bC$) is a consistent estimator of $\bR \bH_R$ ($\bC \bH_C$) and $\hat \bF_t$ is a consistent estimator of $ \bH_R^{-1} \bF_t {\bH_C^{-1}}^\top$.

Let $\bV_{R, pqT} \in \RR^{k \times k}$ and $\bV_{C, pqT} \in \RR^{r \times r}$ be the diagonal matrices consisting of the first $k$ and $r$ largest eigenvalues of $\hat\bM_R=\frac{1}{p q T} \sum_{t = 1}^{T} \tilde\bY_t\tilde\bY_t^\top$ and $\hat\bM_C=\frac{1}{p q T} \sum_{t = 1}^{T} \tilde\bY_t^\top\tilde\bY_t$ in decreasing order, respectively.
By definition of our estimators $\hat\bR$ and $\hat\bC$, we have
\begin{equation*}
\hat\bR = \frac{1}{pqT} \sum_{t=1}^{T} \tilde\bY_t\tilde\bY_t^\top\hat\bR\bV_{R, pqT}^{-1} \quad \text{and} \quad \hat\bC = \frac{1}{pqT}\sum_{t=1}^{T}\tilde\bY_t^\top\tilde\bY_t \hat\bC\bV_{C, pqT}^{-1}.
\end{equation*}
Define $\bH_R \in \RR^{r \times r}$ and $\bH_C \in \mathbb{R}^{r \times r}$ as
\begin{eqnarray}
\bH_R = \frac{1}{p q T} \sum_{t=1}^{T} \tilde\bF_t \bC^\top \bC \tilde\bF_t^\top \bR^\top\hat\bR\bV_{R, pqT}^{-1} \in \mathbb{R}^{k \times k} \label{eqn:def_HR} \\
\bH_C = \frac{1}{p q T} \sum_{t=1}^{T} \tilde\bF_t^\top \bR^\top \bR \tilde\bF_t \bC^\top\hat\bC\bV_{C, pqT}^{-1} \in \mathbb{R}^{r \times r} \label{eqn:def_HC},
\end{eqnarray}
which are bounded as $p, q, T \rightarrow \infty$ (See Appendix
A 
for more details).
Theorem \ref{thm:Rhat_F_norm_convg} shows that $\hbR$ and $\hbC$ converge in Frobenius and $\ell_2$ norm.

\begin{theorem}  \label{thm:Rhat_F_norm_convg}
    Under Assumptions   \ref{assume:alpha-mixing} - \ref{assume:CLT}, we have, as $k$, $r$ fixed and $p, q, T \rightarrow \infty$,
    \[
    \frac{1}{p}\norm{\hat\bR-\bR\bH_R}^2_F = \Op{\frac{1}{\min \braces{p, qT}} }, \quad
    \frac{1}{q}\norm{\hat\bC-\bC\bH_C}^2_F = \Op{\frac{1}{\min \braces{q, pT}}}.
    \]
    Consequently,
    \[\frac{1}{p} \norm{\hat\bR-\bR\bH_R}^2= \Op{\frac{1}{\min \braces{p, qT}} }, \quad \frac{1}{q}\norm{\hat\bC - \bC\bH_C}^2= \Op{\frac{1}{\min \braces{q, pT}}}.\]
\end{theorem}

\begin{remark} \label{remark:thm1-1}
    In the vectorized model~\eqref{eqn:mfm-vec}, we denote $\bLambda = \bC\otimes\bR$.
    Applying results in \cite{bai2002determining} and \cite{bai2003inferential}, we obtain $\frac{1}{pq} \norm{\hat\bLambda-\bLambda\bH}^2= \Op{\frac{1}{\min \braces{pq, T}} }$ where $\bH\in\RR^{kr\times kr}$ is an orthonormal matrix.
    Theorem \ref{thm:Rhat_F_norm_convg} establishes faster $\ell_2$ convergence rate for both $\hat\bR$ and $\hat\bC$ in a high-dimensional regime where $p, q \ge T$.
    Furthermore, we obtain $\hat \bR$ and $\hat \bC$ directly by applying PCA to \eqref{eqn:hat-MR-aggr} and \eqref{eqn:hat-MC-aggr}, which converge faster than the PCA for
    vectorized model~\eqref{eqn:mfm-vec}.
    In addition, in order to use the tensor structure in the factor loadings, after obtaining $\hat\bLambda$ from the vectorized PCA, one needs to carry out a second step to estimate $\bR$ and $\bC$ from $\hat\bLambda$ which amounts to noisy Kronecker production decomposition.
    See \cite{cai2018rate,wedin1972perturbation,cai2019kopa} and references therein for more discussions on this topic.  Since $\hat\bLambda = (\bC\otimes\bR) \bH + \bW$, where $\bW$ is the estimation error in the first step that are dependent across entries, it is not clear how the second step aggregates biases and reduce variances.
\end{remark}

\begin{remark}
    The present paper considers only the fixed $k$ and $r$, which is common in factor analysis.
    The case with growing $k$ and $r$ can be obtained by book-keeping all the $k$ and $r$ in the proofs.
    See \cite{fan2020statistical} and Appendix B  of \cite{chen2020community} for results on growing $k$ and $r$ in the vector factor model setting.
\end{remark}

Before presenting our main theorem on the asymptotic normality, we define several quantities that are used in the theorem.
Letting $\bmu_F = \E{\bF_t}$ and
\begin{equation} \label{eqn:Sigma_FC_FR}
\bSigma_{FC} \defeq \E{\paran{\bF_t-\bmu_F} (\bC^\top \bC/q) \paran{\bF_t-\bmu_F}^\top},\quad\text{and}\quad
\bSigma_{FR} \defeq \E{\paran{\bF_t-\bmu_F}^\top (\bR^\top\bR/p) \paran{\bF_t-\bmu_F}},
\end{equation}
then
\begin{equation} \label{eqn:Sigma_FC_FR_tilde}
    \begin{aligned}
        \tilde\bSigma_{FC} \defeq
        \frac{1}{q}\,\E{\tilde\bF_t\bC^\top\bC \tilde\bF_t^\top} =  \bSigma_{FC}+\paran{\alpha+1}\frac{1}{q}\bmu_F\bC^\top\bC\bmu_F^\top, \\
        \tilde\bSigma_{FR} \defeq
        \frac{1}{p}\,\E{\tilde\bF_t^\top\bR^\top\bR\tilde\bF_t} = \bSigma_{FR}+\paran{\alpha+1}\frac{1}{p}\bmu_F^\top\bR^\top\bR\bmu_F.
    \end{aligned}
\end{equation}

Consider again the special case where $\bF_t \sim \calM \calN(\bmu_F, \bU_F, \bV_F)$. Then,  $\bF_t \bC^\top \sim \calM \calN(\bmu_F\bC^\top, \bU_F, \bC \bV_F \bC^\top)$, $\bR \bF_t \sim \calM \calN(\bR\bmu_F, \bR \bU_F \bR^\top, \bV_F)$, and
\begin{equation*}
\begin{split}
\bSigma_{FC} = \bU_F \cdot \Tr{ \bV_F \frac{\bC^\top \bC}{q}}, \quad \tilde\bSigma_{FC} = \bU_F \cdot \Tr{ \bV_F \frac{\bC^\top \bC}{q}} + \paran{\alpha+1}\frac{1}{p}\bmu_F\bR^\top\bR\bmu_F^\top. \\
\bSigma_{FR} = \bV_F \cdot \Tr{\bU_F \frac{\bR^\top \bR}{p}}, \quad
\tilde\bSigma_{FR} = \bV_F \cdot \Tr{\bU_F \frac{\bR^\top \bR}{p}} + \paran{\alpha+1}\frac{1}{p}\bmu_F\bR^\top\bR\bmu_F.
\end{split}
\end{equation*}
Matrix $\bSigma_{FC}$ can be interpreted as the row covariance of $\bF_t$ scaled by the strengths of column variances of $\bF_t \bC^\top$ and $\bSigma_{FR}$ can be interpreted as the column covariance of $\bF_t$ scaled by the strengths of row variances of $\bR \bF_t^\top$.
Matrices $\bSigma_{FC}$ and $\bSigma_{FR}$ contain the aggregated information of moments of rows of $\bF\bC^\top$ and $\bF^\top\bR$, respectively.

Theorem \ref{thm:asymp normality of R/C} establishes that $\hat \bR$ and $\hat \bC$ are good estimators of $\bR \bH_R$ and $\bC \bH_C$, respectively, and each row of $\hat \bR - \bR \bH_R$ and $\hat \bC - \bC \bH_C$ are asymptotically normal.
The following assumption on eigenvalues is needed.

\begin{assumption} \label{assume:eigen-val}
    The eigenvalues of the $k \times k$ matrix $\bOmega_R \tilde\bSigma_{FC}$ are distinct and so are
    the eigenvalues of the $r \times r$ matrix $\bOmega_C \tilde\bSigma_{FR}$.
\end{assumption}

\begin{theorem} \label{thm:asymp normality of R/C}
    Under Assumptions \ref{assume:alpha-mixing}-\ref{assume:eigen-val}, as $k$, $r$ fixed and $p, q, T \rightarrow \infty$, we have:
    \begin{enumerate}[label=(\roman*)]
        \item For \textbf{row loading matrix $\bR$}, if $\sqrt{q T} / p \rightarrow 0$, then
        \begin{equation*}
        \sqrt{qT} \left( \hat\bR_{i \cdot} - \bH_R^\top \bR_{i \cdot} \right) =\bV_{R, pqT}^{-1}\cdot\frac{\hat \bR^\top \bR }{p}\cdot\frac{1}{\sqrt{qT}} \sum_{t=1}^{T} \tilde\bF_t\bC^\top\tilde\bE_{t, i \cdot} + \op{1}
        \convdist \calN\paran{\bzero, \bSigma_{R_i}},
        \end{equation*}
        where
        \begin{equation} \label{eqn:R-asym-cov}
        \bSigma_{R_i} \defeq \bV_R^{-1}\bQ_R
        \paran{\bPhi_{R,i,11}+\alpha\bPhi_{R,i,12}\bmu_F^\top+\alpha\bmu_F\bPhi_{R,i,21}+\alpha^2\bmu_F\bPhi_{R,i,22}\bmu_F^\top}
        \bQ_R^\top\bV_R^{-1},
        \end{equation}
        and
        \begin{equation}
            \begin{aligned}
                \bPhi_{R, i, 11} & = \underset{q, T \rightarrow \infty}{{\rm p lim}} \frac{1}{qT}  \sum_{t=1}^{T} \sum_{s=1}^{T} \E{\bF_t \bC^\top \be_{t, i \cdot} \be_{s, i \cdot}^\top \bC \bF_s^\top}, \\
                \bPhi_{R, i, 12} & = \bPhi_{R, i, 21}^\top = \underset{q, T \rightarrow \infty}{{\rm p lim}} \frac{1}{qT}  \sum_{t=1}^{T} \sum_{s=1}^{T} \E{\bF_t \bC^\top \be_{t, i \cdot} \be_{s, i \cdot}^\top \bC}, \\
                \bPhi_{R, i, 22} & = \underset{q, T \rightarrow \infty}{{\rm p lim}} \frac{1}{qT}  \sum_{t=1}^{T} \sum_{s=1}^{T} \E{\bC^\top \be_{t, i \cdot} \be_{s, i \cdot}^\top \bC}.
            \end{aligned}
        \end{equation}
        Matrix $\bQ_R \defeq \bV_R^{1/2}\;\bPsi_R^\top\;\tilde\bSigma_{FC}^{- 1/2}$
        where $\bV_R$ is a diagonal matrix whose entries are the eigenvalues of $\tilde\bSigma_{FC}^{1/2}\;\bOmega_{R}\;\tilde\bSigma_{FC}^{1/2}$ in decreasing order,  $\bPsi_R$ is the corresponding eigenvector matrix such that $\bPsi_R^\top \bPsi_R = \bI$, $\bOmega_{R}$ defined in Assumption \ref{assume:loading_matrix} and $\tilde\bSigma_{FC}$ is defined in \eqref{eqn:Sigma_FC_FR_tilde}.

        \item For \textbf{column loading matrix $\bC$}, if $\sqrt{p T} / q \rightarrow 0$, then
        \begin{equation*}
        \sqrt{pT}\paran{\hat\bC_{j\cdot}-\bH_C^\top\bC_{j \cdot}}
        =\bV_{C, pqT}^{-1} \frac{\hat \bC^\top \bC }{q} \frac{1}{\sqrt{pT}} \sum_{t=1}^{T} \bF_t^\top \bR^\top \bE_{t, \cdot j } + \op{1}
        \convdist \calN\paran{\bzero, \bSigma_{C_j}},
        \end{equation*}
        where
        \begin{equation} \label{eqn:C-asym-cov}
        \bSigma_{C_j} \defeq \bV_C^{-1}\bQ_C
        \paran{\bPhi_{C,j,11}+\alpha\bPhi_{C,j,12}\bmu_F+\alpha\bmu_F^\top\bPhi_{C,j,21}+\alpha^2\bmu_F^\top\bPhi_{C,j,22}\bmu_F}
        \bQ_C^\top \bV_C^{-1},
        \end{equation}
        and $ \bPhi_{C,j,11}$, $\bPhi_{C,j,12}$ and  $ \bPhi_{C,j,22}$ are defined similarly to $ \bPhi_{R,i,11}$, $\bPhi_{R,i,12}$ and  $ \bPhi_{R,i,22}$.
        Matrix $\bQ_C \defeq \bV_C^{1/2}\;\bPsi_C^\top\;\tilde\bSigma_{FR}^{- 1/2}$
        where
        $\bV_C$ is a diagonal matrix whose entries are the eigenvalues of $\tilde\bSigma_{FR}^{1/2}\;\bOmega_{C}\;\tilde\bSigma_{FR}^{1/2}$ in decreasing order,  $\bPsi_C$ is the corresponding eigenvector matrix such that $\bPsi_C^\top\bPsi_C = \bI$, $\bOmega_{C}$ is defined in Assumption \ref{assume:loading_matrix}, and $\tilde\bSigma_{FR}$ is defined in \eqref{eqn:Sigma_FC_FR_tilde}.
    \end{enumerate}
\end{theorem}

Note that the asymptotic variance depends on $\alpha$ in a quadratic form and its minimum typically exists.
In particular, if $\bPhi_{R,i,12}=0$ and $\bPhi_{C,i,12}=0$,
the linear term is zero and hence $\alpha_{opt} = 0$.  In this case, $\alpha$-PCA outperforms the convention 2D-PCA, which takes $\alpha = -1$.

\begin{remark} \label{remark:thm2}
    {\em (Optimal $\alpha$ based on different criteria.)}
    Scalar $\alpha$ is a hyper-parameter used in the estimation to balance the information of the first and second moments.
    When $\alpha=-1$, $\alpha$-PCA uses only the second moment and reduces to the $2D$-PCA algorithm.
    Theorems 1 and 2 show that the convergence rates of $\hat\bR_{i\cdot}$ and $\hat\bC_{j\cdot}$ are not affected by $\alpha$.
    However, the asymptotic variances in \eqref{eqn:R-asym-cov} and \eqref{eqn:C-asym-cov} are dependent on the value of $\alpha$.
    Thus, the asymptotic variances of $\hat\bR_{i\cdot}$ and $\hat\bC_{j\cdot}$ can be used as a criterion to find the optimal $\alpha$.

    When $\bmu_f=\bzero$, \eqref{eqn:R-asym-cov} and \eqref{eqn:C-asym-cov}  show that the value of $\alpha$ does not affect the asymptotic variance.
    Indeed, in this case, the first moments do not provide any extra information.
    When $\bmu_f\ne\bzero$, one criterion is to minimize $p^{-1}\sum_{i=1}^p \Tr{\bSigma_{R_i}}$, which controls the asymptotic variance in an average sense.
    We can obtain an analytical form of $\alpha_{opt}$ as
    \begin{equation} \label{eqn:rough-alpha}
    \alpha_{opt} = - \frac{1}{2} \Tr{ \bmu_F^\top\bPhi_{R,22}\bmu_F }^{-1} \Tr{ \bPhi_{R,12}\bmu_F+\bmu_F^\top\bPhi_{R,21} },
    \end{equation}
    where $\bPhi_{R,kl} = p^{-1}\sum_{i=1}^p\bPhi_{R,i,kl}$ for $k,l=1,2$.
    If  $\bPhi_{R,21} = \bPhi_{R,12} =0$, then $\alpha_{opt} = 0$ for the criterion of minimizing $p^{-1}\sum_{i=1}^p \Tr{\bSigma_{R_i}}$.
    In this case, aggregation indeed gains, putting equal weights on both the first and the second moments.
    The simulation in Section \ref{sec:simul-alpha} confirms this theoretical result.

    For other criterion based on asymptotic variances such as $\underset{i\in[p]}{\max}\paran{\bSigma_{R_i}}$, an analytical form of $\alpha$ does not exist.
    However, we are still able to use computational methods to search for the optimal $\alpha$ that minimize the criterion as a function of $\bSigma_{R_i}$ and $\bSigma_{C_j}$ based on \eqref{eqn:R-asym-cov} and \eqref{eqn:C-asym-cov}.
\end{remark}

\begin{remark}
    {\em (Practical guidance for choosing $\alpha$.)}
    As discussed above, the optimal choice of $\alpha$ can be chosen according to \eqref{eqn:rough-alpha} for the purpose of minimizing the asymptotic variance.
    If one decides to seek for a better choice, one can  search $\alpha$ over a grid of points for the one that optimizes an application-specific criterion.
    For example, in Section \ref{sec:application_macro_indices} with multinational macroeconomic indices, we would like the variance of estimators to be minimal.
    So we find optimal $\alpha$ as one that minimizes the trace $\Tr{\hat\bSigma_R}$ where $\hat\bSigma_R= p^{-1}\sum_{i=1}^p \hat\bSigma_{R_i}$.
    This value can be calculated according to equation \eqref{eqn:cov-est} for a grid of $\alpha$'s, as plotted in Figure \ref{fig:oecd-traces-alpha}.
    Alternatively, in Section \ref{sec:application_image} with image data set, we care most about the reconstruction error which is measured by the ratio between residual sum of squares over the total sum of squares (RSS/TSS).
    So we search the optimal $\alpha$ that minimize the RSS/TSS over a grid of $\alpha$s, as plotted in Figure \ref{fig:appl-alpha}.
\end{remark}

\begin{theorem} \label{thm:asymp_Ft}
    Under Assumptions \ref{assume:alpha-mixing}-\ref{assume:eigen-val}, as $k$, $r$ fixed and $p, q, T \rightarrow \infty$, we have 
    \[
    \hat \bF_t - \bH_R^{-1} \bF_t {\bH_C^{-1}}^\top = \Op{\frac{1}{\min(p, q)}}.
    \]
\end{theorem}

\begin{theorem} \label{thm:asymp_St}
    Under Assumptions \ref{assume:alpha-mixing}-\ref{assume:eigen-val}, as $k$, $r$ fixed and $p, q, T \rightarrow \infty$, we have  the following convergence result of the estimator \eqref{eqn:S_t_est} of the signal part $\bS_t = \bR \bF_t \bC^\top$.
    \[
    \hat \bS_{t,ij} - \bS_{t,ij} =  \Op{\frac{1}{\min \paran{p, q, \sqrt{pT}, \sqrt{qT}}}},  \quad\text{for any } 1 \le i \le p \text{ and } 1 \le j \le q.
    \]
\end{theorem}

\begin{remark}
Theorems \ref{thm:asymp_Ft} does not require any restriction on the relationship between $p$, $q$ and $T$ except that they all go to infinity.
Theorems \ref{thm:asymp_Ft} and \ref{thm:asymp_St} show that, in order to estimate the latent factor $\bF_t$ and signal $\bS_t$ consistently, we need to have dimensions $p$ and $q$ approach infinity. An explanation is that we need to have sufficient information to distinguish the signal $\bR \bF_t \bC^\top$ from the noise $\bE_t$ at each time point $t$.
Theorems \ref{thm:asymp normality of R/C}, \ref{thm:asymp_Ft} and \ref{thm:asymp_St} present the asymptotic properties when the dimension of the latent matrix factor $k \times r$ is assumed to be known.
\end{remark}


\section{Estimating Covariance Matrices} \label{sec:cov-est}

In this section, we derive consistent estimators of the asymptotic variance-covariance matrices.
According to Theorem \ref{thm:asymp normality of R/C}, the asymptotic covariance of $\hat \bR_{i \cdot}$, $1 \le i \le p$, is given by
\begin{equation} \label{eqn:R-finite-cov}
\bSigma_{R_i}
= \bV_{R, pqT}^{-1} \bQ_R
\begin{pmatrix} \bI_k & \alpha\bmu_F \end{pmatrix}
\begin{pmatrix}
\bPhi_{R, i, 11} & \bPhi_{R, i, 12} \\
\bPhi_{R, i, 21} & \bPhi_{R, i, 22}
\end{pmatrix}
\begin{pmatrix} \bI_k \\ \alpha\bmu_F^\top \end{pmatrix}
\bQ_R^\top \bV_{R,pqT}^{-1}.
\end{equation}
Term $\bV_{R, pqT}$ is estimated as the $k \times k$ diagonal matrix of the first $k$ largest eigenvalues of $\frac{1}{p q T} \sum_{t = 1}^{T}\tilde\bY_t \tilde\bY_t^\top$ in decreasing order.
To estimate the middle term sandwiched by $\bV_{R, pqT}^{-1}$, we use the heteroskedasticity and autocorrelation consistent (HAC) estimators \citep{newey1987simple} based on series $\braces{\hat\bF_t, \; \hat\bC^\top, \; \hat\be_{t, i \cdot}}_{t\in[T]}$ where $\hat\bF_t$ and $\hat\bC$ are estimated in Section \ref{sec:estimation} and $\hat\bE_t=\bY_t-\hat\bR \hat\bF_t \hat\bC^\top$.
Specifically, for a tuning parameter $m$ that satisfies
and $m \to \infty$ and  $m / \paran{qT}^{1/4} \longrightarrow 0$, it is defined as
\[
\bD_{R, 0,i} + \sum_{\nu = 1}^{m} \paran{1 - \frac{\nu}{1 + m}} \paran{\bD_{R, \nu, i} + \bD_{R, \nu, i}^\top},
\]
where
\[
\bD_{R,\nu,i}=
\begin{pmatrix} \bI_k & \alpha\bar{\hat\bF} \end{pmatrix}
\begin{pmatrix}
\frac{1}{qT}\sum_{t=1+\nu}^{T} \hat \bF_t \hat\bC^\top \hat\be_{t, i \cdot} \hat \be_{t-\nu, i \cdot}^\top \hat \bC \hat \bF_{t-\nu}^\top & \frac{1}{qT} \sum_{t=1+\nu}^{T} \hat \bF_t \hat\bC^\top \hat\be_{t, i \cdot} \hat\be_{t-\nu,i\cdot}^\top \hat\bC \\
\frac{1}{qT}\sum_{t=1+\nu}^{T} \hat\bC^\top\hat\be_{t,i\cdot}\hat\be_{t-\nu, i\cdot}^\top\hat\bC\hat \bF_{t-\nu}^\top &
\frac{1}{qT}\sum_{t=1+\nu}^{T}\hat\bC^\top\hat\be_{t,i\cdot}\hat\be_{t-\nu,i \cdot}^\top\hat\bC
\end{pmatrix}
\begin{pmatrix} \bI_k \\ \alpha\bar{\hat\bF}^\top \end{pmatrix},
\]
and $\bar{\hat\bF}=\frac{1}{T}\sum_{t=1}^{T}\hat\bF_t$ is the estimated mean.
While a HAC estimator based on true $\braces{\bF_t, \; \bC^\top, \; \be_{t, i \cdot}}_{t\in[T]}$, a HAC estimator based on $\braces{\hat\bF_t, \; \hat\bC^\top, \; \hat\be_{t, i \cdot}}_{t\in[T]}$ is estimating $\bQ_R \bPhi_{R, i} \bQ_R^\top$ because $\hat \bF_t$ estimates $\bH_R^{-1} \bF_t {\bH_C^\top}^{-1}$, $\hat \bC$ estimates $\bC \bH_C$ and $\bar{\hat\bF}$ estimates $\bH_R^{-1} \bmu_F {\bH_C^\top}^{-1}$.
Thus, a HAC estimator of the covariance of $\bSigma_{R_i}$ is given by
\begin{equation}
\label{eqn:cov-est}
\hat \bSigma_{R_i} = \bV_{pqT, R}^{-1} \left( \bD_{R, 0,i} + \sum_{\nu = 1}^{m} \paran{1 - \frac{\nu}{1 + m}} \paran{\bD_{R, \nu, i} + \bD_{R, \nu, i}^\top} \right) \bV_{pqT, R}^{-1}
\end{equation}
Similar for $\hat \bC_{j \cdot}$, $1 \le j \le q$,  a HAC estimator of the covariance is given by
\[
\hat \bSigma_{C_j} = \bV_{pqT, C}^{-1}  \left( \bD_{C, 0,j} + \sum_{\nu = 1}^{m} \paran{1 - \frac{\nu}{1 + m}} \paran{\bD_{C, \nu, j} + \bD_{C, \nu, j}^\top} \right) \bV_{pqT, C}^{-1},
\]
where
\[
\bD_{C, \nu, j} =
\begin{pmatrix} \bI_r & \alpha\bar{\hat\bF}^\top \end{pmatrix}
\begin{pmatrix}
\frac{1}{pT}\sum_{t=1+\nu}^{T} \hat\bF_t^\top \hat\bR^\top \hat\be_{t,\cdot j} \hat \be_{t-\nu,\cdot j}^\top \hat\bR\hat\bF_{t-\nu} & \frac{1}{pT} \sum_{t=1+\nu}^{T} \hat\bF_t^\top \hat\bR^\top \hat\be_{t,\cdot j} \hat\be_{t-\nu,\cdot j}^\top \hat\bR \\
\frac{1}{pT}\sum_{t=1+\nu}^{T} \hat\bR^\top\hat\be_{t,\cdot j}\hat\be_{t-\nu,\cdot j}^\top\hat\bR\hat \bF_{t-\nu} &
\frac{1}{pT}\sum_{t=1+\nu}^{T}\hat\bR^\top\hat\be_{t,\cdot j}\hat\be_{t-\nu,\cdot j}^\top\hat\bR
\end{pmatrix}
\begin{pmatrix} \bI_r \\ \alpha\bar{\hat\bF} \end{pmatrix},
\]
and $\bar{\hat\bF}=\frac{1}{T}\sum_{t=1}^{T}\hat\bF_t$ is the estimated mean.
The following theorem confirms the consistency.

\begin{theorem} \label{thm:cov-est-consistent}
    Under Assumptions \ref{assume:alpha-mixing}-\ref{assume:eigen-val},  as $k$, $r$ fixed and $p, q, T \rightarrow \infty$, $\hat \bSigma_{R_i}$ and $\hat \bSigma_{C_j}$ are consistent for $\bSigma_{R_i}$ and $\bSigma_{C_j}$, respectively.
\end{theorem}

\section{Simulation} \label{sec:simulation}

In this section, we use Monte Carlo simulations to assess the adequacy of the asymptotic results in approximating the finite sample distributions of $\hat \bR_{i\cdot}$ and $\hat \bC_{j\cdot}$ and the convergence rate of $\bF_t$.
We only report the result for $\hat\bR_{i\cdot}$ and $\bF_t$ because $\hat \bC_{j\cdot}$ shares similar properties to $\hat\bR_{i\cdot}$.

\subsection{Settings}

Throughout, the matrix observations $\bY_t$'s are generated according to model~\eqref{eqn:mfm}.
The dimension of the latent factor matrix $\bF_t$ is fixed at $k \times r = 3 \times 3$.
The values of $p$, $q$, and $T$ vary in different settings.
The true loading matrices $\bR$ and $\bC$ are independently sampled from the uniform distribution $\calU \left( -1, 1 \right)$.
The latent factor and noise matrices are allowed to be dependent across rows, columns or time, respectively, in different settings to be specified later.

We present the following results under different settings in the subsequent subsections.
We refer our method and the one proposed in \cite{wang2019factor} as $\alpha$-aggregated PCA ($\alpha$-PCA) and auto-covariance based PCA (AC-PCA), respectively.
Results 1-3 compare specifically the results obtained by $\alpha$-PCA with those by AC-PCA.
Result 4 presents the results obtained by $\alpha$-PCA with different values of $\alpha$.
Result 5 illustrates the optimal choice of the hyper-parameter $\alpha$.
\begin{enumerate}[itemsep=-1ex, label=\textit{Result \arabic*.}, leftmargin=*]
    \item \textit{(Estimating latent dimensions.)} The latent dimensions are estimated by the eigen-ratio method of (\ref{eqn:eigen-ratio}). Results are presented in tables of frequencies of $\hat{k} \times \hat{r}$.
    \item \textit{(Proposition \ref{thm:Rhat_F_norm_convg}: Convergence of $\hat \bR, \hat \bC$.)} We report box plots of the ratios between space distances $\calD(\hat \bR, \bR)$ (defined in \eqref{eqn:spdist}) retrieved from $\alpha$-PCA and those from AC-PCA.
    \item \textit{(Theorem \ref{thm:asymp_Ft}: Convergence of  $\hat \bF_t$.)} To demonstrate that $\hat \bF_t$ is estimating a transformation of $\bF_t$ for $t\in [T]$, we compute the $\bH_R$ and $\bH_C$ according to (\ref{eqn:def_HR}) and (\ref{eqn:def_HC}), respectively, and report box plots of $\norm{\hat \bF_t - \bH_R^{-1} \bF_t \bH_C^{-1\top} }$.
    \item \textit{(Theorem \ref{thm:asymp normality of R/C}: Asymptotic normality $\hat \bR - \bR \bH_R$.)} We first consider the asymptotic distribution of $\hat \bR$. We estimate $\hat{\bf{\Sigma}}_{R_0}$ according to \eqref{eqn:cov-est} and average. Then we compute the $k \times 1$ vectors $\hat{\bf{\Sigma} }_R^{-1/2}(\hat{\bR}_{0,\cdot} - \bH_R^{\top}\bR_{0,\cdot})$ and report 1-dimensional histograms of each first component.
     \item \textit{(Optimal $\alpha$ based on Theorem \ref{thm:asymp normality of R/C}.)}
     For each value of $\alpha$ in $[-1, 5]$ with a step-size of $0.1$, we calculate the covariance matrix $\hat\bSigma_{R_0}$ of $\hat\bR_{0\cdot}$ according to \eqref{eqn:cov-est}.
     The empirical optimal $\alpha$ is very close to the theoretical value given in \eqref{eqn:rough-alpha}.
     See Section 5.4 for details.
 \end{enumerate}

\subsection{Comparison of convergence}

In this section, we consider the the finite sample convergence of $\hat \bR_{i\cdot}$, $\hat \bC_{j\cdot}$ and $\bF_t$.
We choose $\left( p,q \right)$ among $\left( 20,20 \right)$, $\left( 20,100 \right)$, or $\left( 100, 100 \right)$ and let $T = 0.5pq$, $pq$, $1.5pq$, or $2pq$, similar to the setup in \cite{wang2019factor}.
For the AC-PCA estimator, we will use lag parameter $h_0 = 1$ since we will be considering uncorrelated models or VAR(1) processes only.
We use the column space distance
\begin{equation} \label{eqn:spdist}
\calD\paran{\bA, \hat\bA} = \norm{\hat\bA\paran{\hat\bA^\top\hat\bA}^{-1}\hat\bA^\top - \bA\paran{\bA^\top\bA}^{-1}\bA^\top},
\end{equation}
for any rank $k$ matrices $\hat\bA,\bA\in\RR^{p\times k}$.
To keep things simple, we only use the second moment information, that is $\alpha = -1$, in this section.
From Theorems \ref{thm:Rhat_F_norm_convg} and \ref{thm:asymp_Ft}, values of $\alpha$ does not affect the convergence rate in the strong factor regime.
Results in this section are based on $100$ repetitions, which are sufficient as shown in the reported standard deviations.

We simulate data and estimations under three settings as follows.
\begin{enumerate}[topsep=0pt, itemsep=-1ex, label=\textit{(\Roman*)}, leftmargin=*]
    \item \textit{(Uncorrelated.)} \label{case:I} The entries of both $\bF_t$ and $\bE_t$ are uncorrelated across time, rows and columns. Specifically, we simulate temporally independent $\bF_t \sim \calM \calN_{3 \times 3} \left( \bzero, \bI, \bI \right)$ and $\bE_t \sim \calM \calN_{p \times q} \left( \bzero, \bI, \bI \right)$.
    \item \textit{(Weakly correlated cross time.)} The entries of $\bF_t$ and $\bE_t$ are uncorrelated across rows and columns, but weakly correlated temporally. Specifically, we simulate $\vect{\bF_t}$ from the following Vector Auto-Regressive model of order one (VAR(1) model):
    \[
    \vect{\bF_t} = \bPhi \cdot \vect{\bF_{t-1}} + \bepsilon_t,
    \]
    where the AR coefficient matrix $\bPhi = 0.1 \cdot \bI_6$ and $\Var{\bepsilon_t} = 0.99\cdot\bI_9$. Thus, $\Var{\vect{\bF_t}} = \bI_9$. We simulate noise $\bE_t$ also from VAR(1),
    \[
    \vect{\bE_t} = \bPsi \cdot \vect{\bE_{t-1}} + \bu_t,
    \]
    where $\bPsi = \psi \cdot \bI_{pq}$ and $\Var{\bu_t} = 1 - \psi^2$. Thus, $\Var{\vect{\bE_t}} = \bI_{pq}$. We choose $\psi = 0.1$ and then increase to $\psi = 0.5$ to examine how temporal dependence may affect our results. Note that setting (II) with $\psi = 0$ corresponds to setting (I).
    \item \textit{(Weakly correlated cross rows or columns.)}The entries of $\bF_t$ and $\bE_t$ are temporally uncorrelated, but $\bE_t$ is weakly correlated across rows and columns. Specifically, we simulate temporally independent $\bF_t \sim \calM \calN_{3 \times 3} \left( \bzero, \bI, \bI \right)$ and $\bE_t \sim \calM \calN_{p \times q} \left( \bzero, \bU_E, \bV_E \right)$, where $\bU_E$ and $\bV_E$ both have 1's on the diagonal, while have $1/p$ and $1/q$ off-diagonal, respectively.
    Note that Setting (III) correspond to setting (I) when $\Psi = 0$ and the variance of $\bu_t$ are specified as $\bV_E \otimes \bU_E$.
\end{enumerate}

For both latent dimension estimation and convergence results, $\alpha$-PCA consistently converges faster with lower variance and estimates more accurately than AC-PCA over all chosen settings, including a special case in Setting (II) where we increase $\psi$, the strength of temporal correlation.
Thus it is implied that $\alpha$-PCA has significant advantages over AC-PCA when $\bF_t$ and $\bE_t$ are uncorrelated or weakly correlated across rows and columns or time.
In the sequel, we report results for latent dimension, loading matrices and factor matrices under Setting (II) with  $\psi = 0.1$ and $\psi = 0.5$.
Results under setting (I) and (III) are similar and relegated to Appendix C.

\paragraph{Accuracy of estimating unknown dimensions.} We present the frequencies of estimated $(\hat k, \hat r)$ pairs for Setting (II) with $\psi = 0.1$ and $\psi = 0.5$ in Table \ref{table:estim_k_II_1} and \ref{table:estim_k_II_5}, respectively.
In latent dimension estimation, our results demonstrate higher frequencies of correct estimation, and the accuracy increases as $p$, $q$, and $T$ increase.

\begin{table}[htpb!]
    \centering
    \begin{subtable}[c]{\textwidth} 
        \centering
        \caption{Setting II, $\psi = 0.1$. } \label{table:estim_k_II_1}
        \resizebox{\linewidth}{!}{\begin{tabular}{c| c c c c |c c c c | c c c c }
                \multicolumn{5}{c}{$p,q = 20,20$} & \multicolumn{4}{c}{$p,q = 100,20$} & \multicolumn{4}{c}{$p,q = 100,100$}\\ \hline
                ($\hat{k},\hat{r}$)  & $T=.5pq$ & $T=pq$& $T=1.5pq$& $T=2pq$ & $T=.5pq$ & $T=pq$& $T=1.5pq$& $T=2pq$ & $T=.5pq$ & $T=pq$& $T=1.5pq$& $T=2pq$\\ \hline
                ($2,3$) & .075 & .08 & .04 & .03 & 0   & 0   & 0   & 0   & 0 & 0 & 0 &0\\
                \rowcolor[HTML]{EFEFEF} & .025 & .005 & .005 & .015 & 0 & 0 & 0 & 0 & 0 & 0 & 0 &0 \\ \hline
                ($3,2$) & .06  & .05 & .035& .06 & .025& .035& .02 & .045& 0 & 0 & 0 &0\\
                \rowcolor[HTML]{EFEFEF} & .01 & .015 & 0 & .005 & .015 & .005 &.005 & 0 & 0 & 0 &0 &0\\ \hline
                ($3,3$) & .78  &  .8 & .85 & .815& .96 & .95 & .965& .94 & 1 & 1 & 1 &1\\
                \rowcolor[HTML]{EFEFEF} & .955 &  .975 & .995 & .98 & .985 & .995 & .995 & .995 & 1 &1 &1 & 1\\ \hline
                other   & .085  & .07 & .075 & .095 & .015& .015& .015& .015   & 0 & 0 & 0 &0\\
                \rowcolor[HTML]{EFEFEF} & .01 & .005 & 0 & 0 & 0 & 0 & .005 & .005 & 0 & 0 & 0 & 0 \\ \hline
        \end{tabular}}
        \vspace{3mm}
    \end{subtable}
    \begin{subtable}[c]{\textwidth} 
        \centering
        \caption{Setting II, $\psi = 0.5$} \label{table:estim_k_II_5}
        \resizebox{\linewidth}{!}{\begin{tabular}{c| c c c c |c c c c | c c c c }
                \multicolumn{5}{c}{$p,q = 20,20$} & \multicolumn{4}{c}{$p,q = 100,20$} & \multicolumn{4}{c}{$p,q = 100,100$}\\ \hline
                ($\hat{k},\hat{r}$)  & $T=.5pq$ & $T=pq$& $T=1.5pq$& $T=2pq$ & $T=.5pq$ & $T=pq$& $T=1.5pq$& $T=2pq$ & $T=.5pq$ & $T=pq$& $T=1.5pq$& $T=2pq$\\ \hline
                ($2,3$) &.095&.105&.075&.035& 0  & 0  & 0  & 0  & 0 & 0 & 0 &0\\
                \rowcolor[HTML]{EFEFEF} &.025& .03&.005&.015& 0  & 0  & 0  & 0  & 0 & 0 & 0 & 0\\ \hline
                ($3,2$) & .07& .09&.075&.085&.055&.06 &.05 &.11 & 0 & 0 & 0 &0\\
                \rowcolor[HTML]{EFEFEF} &.02 & .02& 0  &.01 & .01& .01& 0  & .01& 0 & 0 & 0 & 0\\ \hline
                ($3,3$) & .66&.615& .71&.685&.895&.875&.92 &.835& 1 & 1 & 1 &1\\
                \rowcolor[HTML]{EFEFEF} &.925&.935&.995&.97 &.985&.995&.995& .99& 1 & 1 & 1 & 1\\ \hline
                other   & .175& .19& .14&.195 & .05  & .065  & .03  & .055  & 0 & 0 & 0 &0\\
                \rowcolor[HTML]{EFEFEF} & .03  & .015  &  .005 &  .005 & .005  & 0  & .005  & 0  & 0 & 0 & 0 & 0 \\ \hline
        \end{tabular}}
    \end{subtable}
    \caption{Table of frequencies of estimated ($\hat{k}, \hat{r}$) pairs estimated by $\alpha$-PCA (highlighted rows) and AC-PCA (not highlighted rows) under Setting II, $\psi = 0.1, 0.5$.
        The truth is $(3,3)$.
    } \label{table:estim_k_II}
\end{table}

\paragraph{Error of loading matrices estimation.}

Figure \ref{fig:space_dist_ratio_II} (a) and (b) show box plots of ratios of the column space distances between $\alpha$-PCA and AC-PCA estimators, under Setting II $\psi = 0.1$ and $\psi = 0.5$ respectively.
Clearly, the estimation errors of $\alpha$-PCA are much smaller than those of AC-PCA, since the ratios are ways below 1.

Detailed numeric values are presented in Table \ref{table:mean_sd_II} which contains the means and standard deviations (in parentheses) of $\calD\paran{\hat\bR, \bR}$, $\calD\paran{\hat\bC, \bC}$ estimated by $\alpha$-PCA (highlighted) and AC-PCA.
All values are multiplied by 10 and rounded.

For the space distances $\calD(\hat \bR, \bR)$, $\calD(\hat \bC, \bC)$, there is a tendency for higher convergence as well as smaller variance at higher $(p,q)$, as well as a slight tendency for better convergence at higher $T$, although the latter effect is less pronounced.
Similar to the space distance results, the $\hat\bF$ convergence also improves as we increase $p,q$, and improves slightly as we increase $T$.

\begin{figure}
    \centering
    \begin{subfigure}[b]{\textwidth}
        \centering
        \caption{Setting II, $\psi = 0.1$. }
        \includegraphics[width=\linewidth,height=\textheight,keepaspectratio=true]{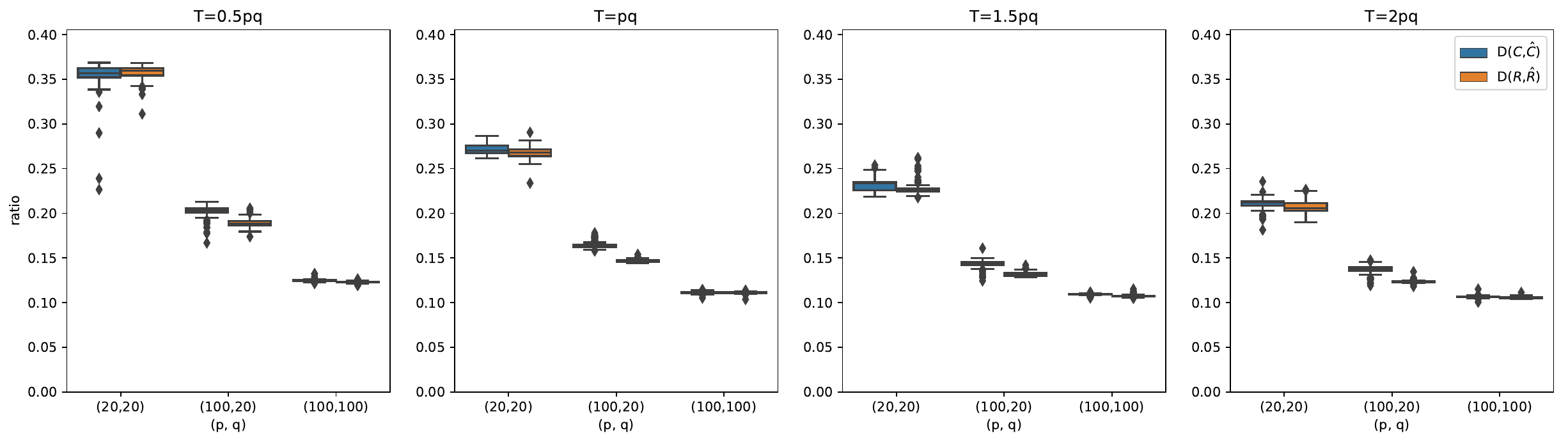}
    \end{subfigure}
    \hfill
    \begin{subfigure}[b]{\textwidth}
        \centering
        \caption{Setting II, $\psi = 0.5$. }
        \includegraphics[width=\linewidth,height=\textheight,keepaspectratio=true]{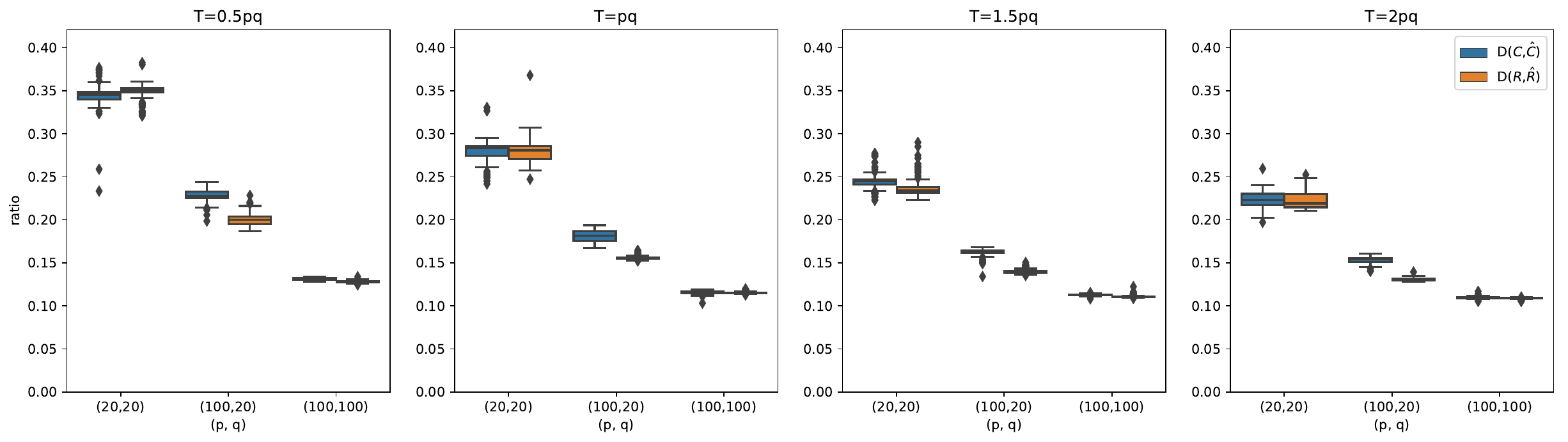}
    \end{subfigure}
    \caption{Box plots of ratios of space distances between $\alpha$-PCA and AC-PCA estimators.
        (a) is under Setting II, $\psi = 0.1$; (b) is under Setting II, $\psi = 0.5$. The estimation errors of $\alpha$-PCA is much smaller than AC-PCA.}
    \label{fig:space_dist_ratio_II}
\end{figure}

\begin{table}[htpb!]
    \centering
    \begin{subtable}[c]{\textwidth} 
        \centering
        \caption{Setting II, $\psi = 0.1$. }
        \resizebox{\linewidth}{!}{
            \begin{tabular}{c c c c c c c c c} \hline
                & \multicolumn{2}{c}{$T = 0.5pq$} & \multicolumn{2}{c}{$T = pq$} & \multicolumn{2}{c}{$T = 1.5pq$} & \multicolumn{2}{c}{$T = 2pq$} \\ \hline
                $(p,q)$  & D($\hat{\bR},\bR$) &  D($\hat{\bC},\bC$) & D($\hat{\bR},\bR$) &  D($\hat{\bC},\bC$) & D($\hat{\bR},\bR$) &  D($\hat{\bC},\bC$) & D($\hat{\bR},\bR$) &  D($\hat{\bC},\bC$)\\ \hline
                \rowcolor[HTML]{EFEFEF}($20,20$) & .40(.08) & .40(.09)  & .29(.07) & .29(.07) & .23(.05) & .23(.05) & .20(.05) & .21(.04) \\
                & 1.12(.24) & 1.14(.31)  & 1.08(.26) & 1.06(.23) & 1.00(.20) & 1.00(.20) & .98(.23) & .98(.18) \\
                \rowcolor[HTML]{EFEFEF}($100,20$) & .14(.01) & .08(.02) & .10(.01) & .05(.02) & .08(.01) & .05(.01) & .07(.01) & .04(.01) \\
                & .76(.06) & .40(.09) & .70(.06) & .35(.07) & .63(.05) & .32(.06) & .58(.05) & .30(.06) \\
                \rowcolor[HTML]{EFEFEF}($100,100$) & .03(.002) & .03(.002) & .02(.002) & .02(.002) & .02(.001) & .02(.001) & .01(.001) & .01(.001)  \\
                & .23(.02) & .23(.02) & .18(.01) & .18(.01) & .15(.01) & .15(.01) & .13(.01) & .13(.01)  \\ \hline
        \end{tabular}}
        \vspace{3mm}
    \end{subtable}
    \begin{subtable}[c]{\textwidth} 
        \centering
        \caption{Setting II, $\psi = 0.5$.}
        \resizebox{\linewidth}{!}{
            \begin{tabular}{c c c c c c c c c} \hline
                & \multicolumn{2}{c}{$T = 0.5pq$} & \multicolumn{2}{c}{$T = pq$} & \multicolumn{2}{c}{$T = 1.5pq$} & \multicolumn{2}{c}{$T = 2pq$} \\ \hline
                $(p,q)$  & D($\hat{\bR},\bR$) &  D($\hat{\bC},\bC$) & D($\hat{\bR},\bR$) &  D($\hat{\bC},\bC$) & D($\hat{\bR},\bR$) &  D($\hat{\bC},\bC$) & D($\hat{\bR},\bR$) &  D($\hat{\bC},\bC$)\\ \hline
                \rowcolor[HTML]{EFEFEF}($20,20$) & .52(.12) & .52(.13)  & .38(.11) & .38(.10) & .29(.07) & .30(.07) & .26(.07) & .27(.06) \\
                & 1.50(.33) & 1.51(.41)  & 1.36(.32) & 1.34(.29) & 1.23(.26) & 1.23(.26) & 1.18(.25) & 1.19(.23) \\
                \rowcolor[HTML]{EFEFEF}($100,20$) & .17(.02) & .11(.02) & .12(.01) & .07(.02) & .10(.01) & .06(.01) & .09(.01) & .05(.01) \\
                & .87(.07) & .46(.10) & .79(.06) & .40(.08) & .72(.06) & .36(.07) & .66(.06) & .34(.07) \\
                \rowcolor[HTML]{EFEFEF}($100,100$) & .03(.003) & .04(.003) & .02(.002) & .02(.002) & .02(.002) & .02(.001) & .02(.001) & .01(.001)  \\
                & .27(.02) & .27(.02) & .21(.02) & .21(.02) & .18(.01) & .18(.01) & .16(.01) & .16(.01)  \\ \hline
        \end{tabular}}
    \end{subtable}
    \caption{Means and standard deviations (in parentheses) of $\calD\paran{\hat\bR, \bR}$, $\calD\paran{\hat\bC, \bC}$ estimated by $\alpha$-PCA (highlighted rows) and AC-PCA (not highlighted rows) under Setting II, $\psi = 0.1, 0.5$. All values multiplied by 10 and rounded.}
    \label{table:mean_sd_II}
\end{table}

\paragraph{Factor matrices estimation errors.}

Figure \ref{fig:F_norm_II} presents the box-plots of the $\ell_2$ norm of the discrepancy between estimated $\hat\bF_t$ and transformed true $\bF_t$, that is temporal-averaged $\norm{\hat\bF_t - \bH_R^{-1}\bF_t{\bH_C^{-1}}^\top}$, under setting II, $\psi = 0.1$ and $0.5$.
As expected, the estimation errors decrease when $p$ or $q$ increases while not affected by $T$.
Results of $\norm{\hat\bF_t - \bH_R^{-1}\bF_t{\bH_C^{-1}}^\top}$ for AC-PCA are not available since \cite{wang2019factor} don't have explicit forms for the rotation matrices $\bH_R$ and $\bH_C$.

\begin{figure}[htpb!]
    \centering
    \includegraphics[width=\linewidth,keepaspectratio=true]{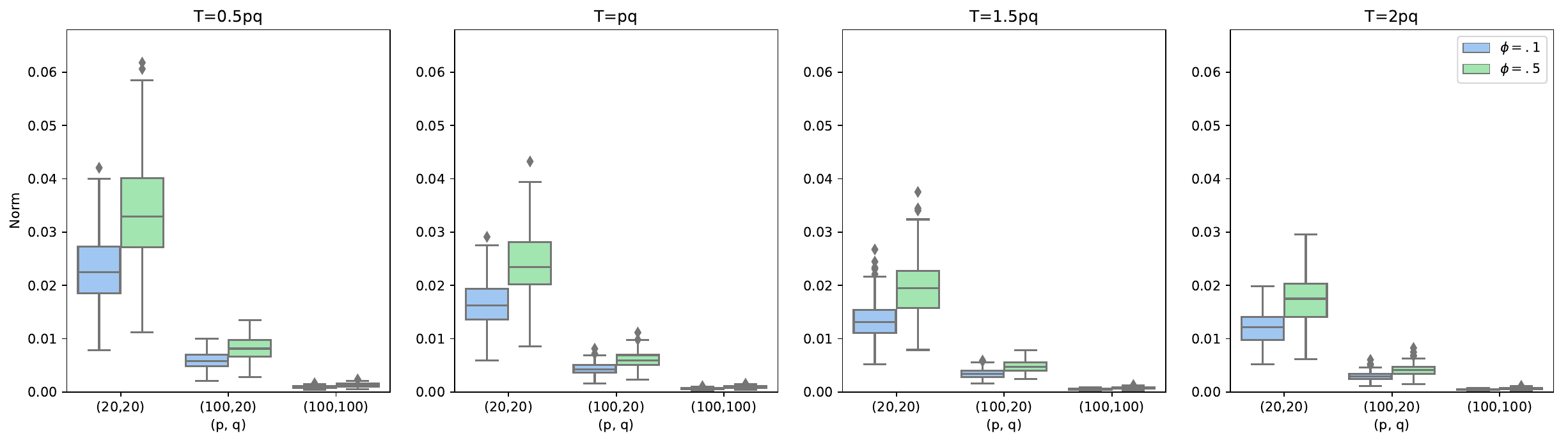}
    \caption{Boxplot of $\norm{\hat\bF_t - \bH_R^{-1}\bF_t{\bH_C^{-1}}^\top}$ under setting II, $\psi = 0.1$ and $0.5$.}
    \label{fig:F_norm_II}
\end{figure}

\subsection{Asymptotic normality} \label{sec:simul-distn}

In this section, we consider the asymptotic normality of the first row of $\hat\bR - \bH_R^\top \bR$ under different values of $\alpha$.
We simulate data under the following setting:
\begin{enumerate}[label=\textit{(\Roman*)}, leftmargin=*]
    \setcounter{enumi}{3}
    \item \textit{($\bF_t$ with non-zero mean.)} The entries of both $\bF_t$ and $\bE_t$ are uncorrelated across time, rows and columns. Specifically, we simulate temporally independent $\bF_t \sim \calM \calN_{3 \times 3} \left( 3\cdot\bI, \bI, \bI \right)$ and $\bE_t \sim \calM \calN_{p \times q} \left( \bzero, \bI, \bI \right)$. \label{case:IV}
\end{enumerate}
According to Theorem \ref{thm:asymp normality of R/C}, the asymptotic normality requires $\sqrt{qT}/p \to 0$ or $\sqrt{pT}/q \to 0$.
Thus we choose $\paran{p,q, T}$ among $\paran{200, 200, 100}$, $\paran{200, 200, 150}$ and $\paran{400, 400, 250}$.
The results for asymptotic normality are based on $1000$ repetitions.
We report results for $p, q, T = 200, 200, 150$ in the main text and the results for the other two settings are relegated to the appendix
Under all settings, the presented QQ plots and histograms demonstrate the asymptotic normality expected from the theorem.

Figure \ref{fig:IV-200-200-150-QQ} presents the QQ plots of the first dimension of the first row of $\hat\bR - \bR\bH_R$ under setting (IV) with $p, q, T = 200, 200, 150$.
Results of the other dimensions are similar.

\begin{figure}[htpb!]
    \centering
    \begin{subfigure}[b]{0.3\textwidth}
        \centering
        \caption*{$\alpha=-1$}
        \includegraphics[width=\linewidth,keepaspectratio=true]{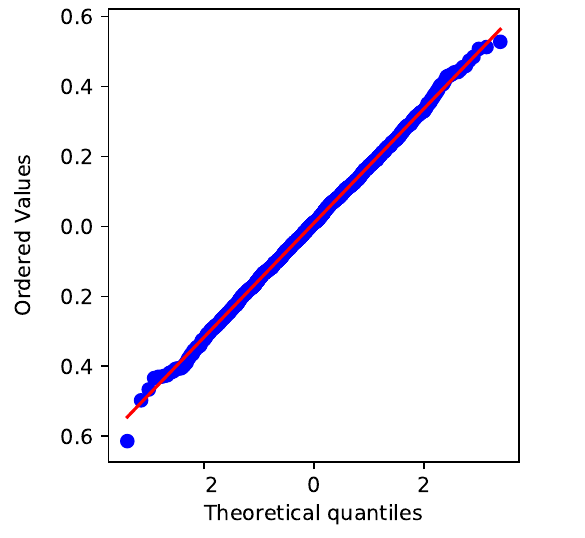}
    \end{subfigure}
    \begin{subfigure}[b]{.3\textwidth}
        \centering
        \caption*{$\alpha=0$}
        \includegraphics[width=\linewidth,keepaspectratio=true]{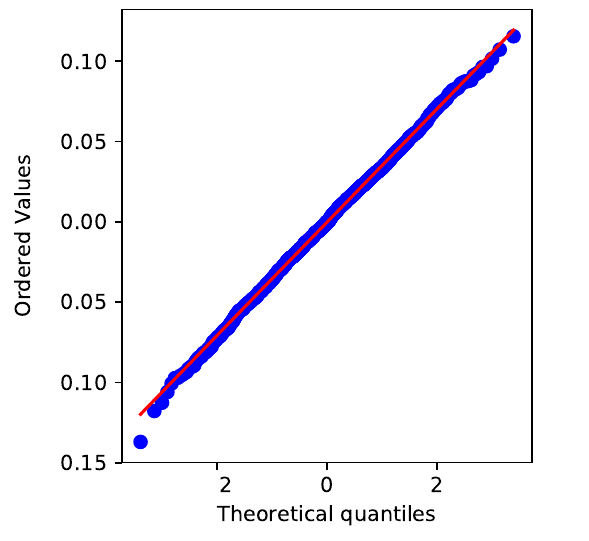}
    \end{subfigure}
    \begin{subfigure}[b]{.3\textwidth}
        \centering
        \caption*{$\alpha=1$}
        \includegraphics[width=\linewidth,keepaspectratio=true]{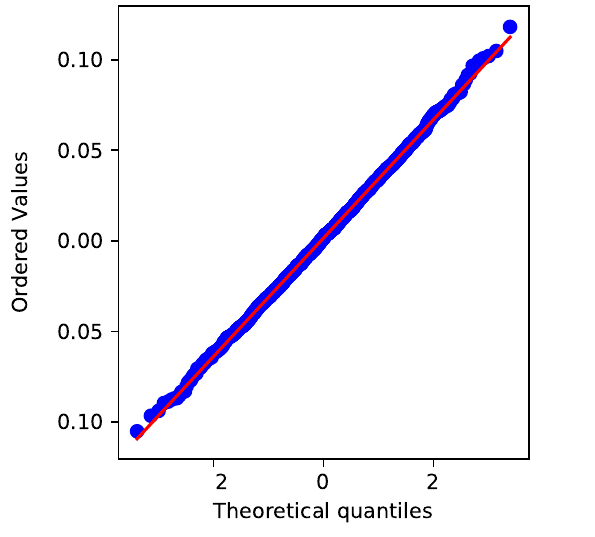}
    \end{subfigure}
    \caption{QQ plots of the first dimension of the first row of $\hat\bR - \bR\bH_R$ with $\alpha=-1$ (left), $0$ (middle) and $1$ (right) under setting (IV) with $p, q, T = 200, 200, 150$.} \label{fig:IV-200-200-150-QQ}
\end{figure}

We calculate the covariance matrix $\hat\bSigma_{R_0}$ of the first row of $\hat\bR - \bR\bH_R$ according to equation \eqref{eqn:cov-est} and plot the histograms of the first dimension of $\hat\bSigma_{R_0}^{-1/2}\paran{\hat\bR_{0\cdot} - \bH_R^\top\bR_{0\cdot}}$ in Figure \ref{fig:IV-200-200-150-dist}.
The plots for other components are similar.

\begin{figure}[htpb!]
    \centering
    \begin{subfigure}[b]{.3\textwidth}
        \centering
        \caption*{$\alpha=-1$}
        \includegraphics[width=\linewidth,keepaspectratio=true]{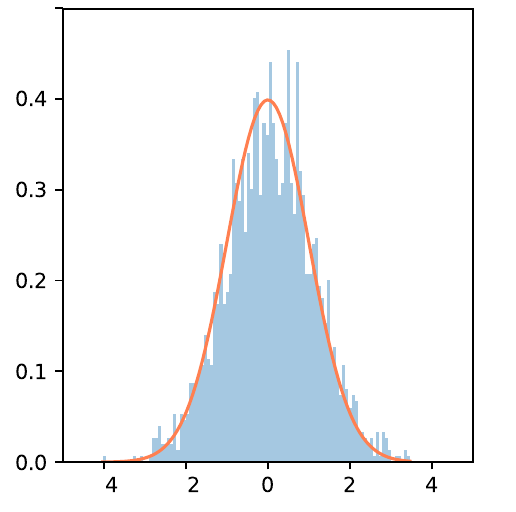}
    \end{subfigure}
    \begin{subfigure}[b]{.3\textwidth}
        \centering
        \caption*{$\alpha=0$}
        \includegraphics[width=\linewidth,keepaspectratio=true]{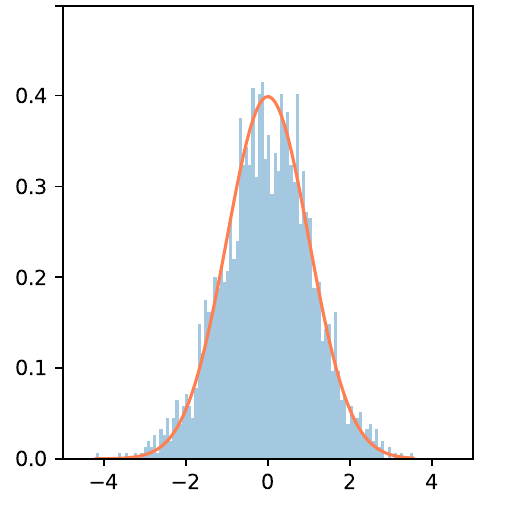}
    \end{subfigure}
    \begin{subfigure}[b]{.3\textwidth}
        \centering
        \caption*{$\alpha=1$}
        \includegraphics[width=\linewidth,keepaspectratio=true]{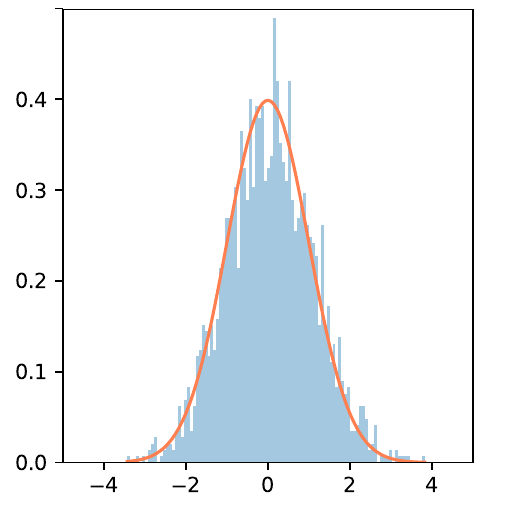}
    \end{subfigure}
    \caption{Histograms of the first dimension of $\hat\bSigma_{R_0}^{-1/2}\paran{\hat\bR_{0\cdot} - \bH_R^\top\bR_{0\cdot}}$ with $\alpha=-1$ (left), $0$ (middle) and $1$ (right) under setting \ref{case:IV} with $p, q, T = 200, 200, 150$. The lines plot the distribution of standard normal distribution.} \label{fig:IV-200-200-150-dist}
\end{figure}

\subsection{ Hyper-parameter selection and optimality of $\alpha$ } \label{sec:simul-alpha}

In this section, we illustrate the optimal choice of the hyper-parameter $\alpha$ on simulated data set.
Specifically, we consider Setting \ref{case:I} and \ref{case:IV} where $\bF_t$ has zero and non-zero means, respectively.
The dimension $\paran{p,q, T}$ is fixed at $\paran{200, 200, 150}$.
The range of $\alpha$ is in $[-1, 5]$ with a step-size of $0.1$.
For each value of $\alpha$, we calculate the covariance matrix $\hat\bSigma_{R_0}$ of $\hat\bR_{0\cdot}$ according to \eqref{eqn:cov-est}.
Figure \ref{fig:alpha} presents the estimation errors and the covariance of the estimator versus different values of $\alpha$.
Under Setting \ref{case:IV} where $\bE_t$ are white noise and independent of $\bF_t$, we know that $\bPhi_{R,i,12} = \bPhi_{C, j, 21}=\bzero$.
The optimal value according to \eqref{eqn:rough-alpha} is $\alpha_{opt} = 0$.
The sample estimation of $\hat\alpha_{opt}$ using \eqref{eqn:rough-alpha} from 200 repetitions has mean $-0.0144$ and standard deviation $0.009$.

Figure \ref{fig:alpha} (a) plots the diagonal elements $\hat\sigma_{R,ii}^2$, $i\in[3]$, and the trace of the covariance matrix $\hat\bSigma_{R_0}$.
The $\alpha$ value corresponding to the dip of all lines are around $\alpha=0$, confirming our calculation of the value of $\alpha$ that minimizing the covariance of estimators.
Although $\alpha$ does not affect the convergence rate in Theorems \ref{thm:Rhat_F_norm_convg} and \ref{thm:asymp_Ft}, Figure \ref{fig:alpha} (b) show that the errors using $\alpha=-1$ is larger under the finite sample setting.

\begin{figure}
    \centering
    \begin{subfigure}[b]{.3\textwidth}
        \centering
        \caption{Setting \ref{case:IV}}
        \includegraphics[width=\linewidth,keepaspectratio=true]{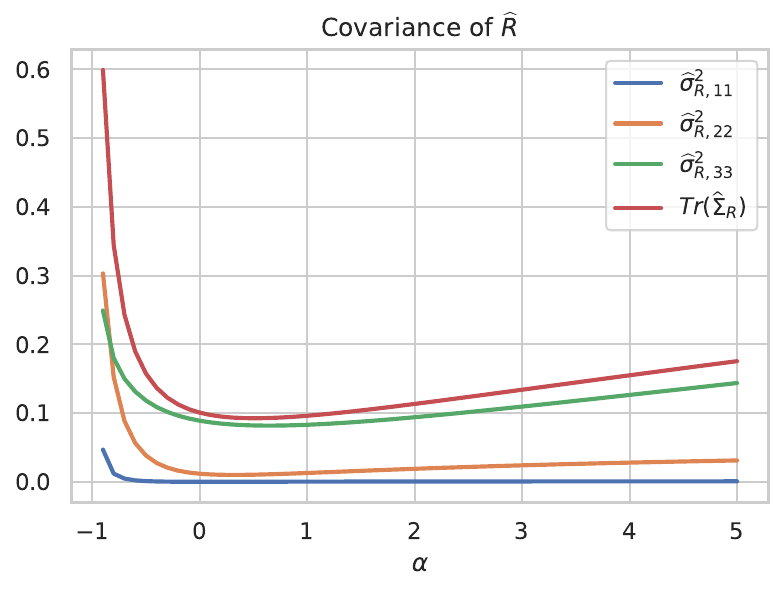}
    \end{subfigure}
    \begin{subfigure}[b]{.3\textwidth}
        \centering
        \caption{Setting \ref{case:IV}}
        \includegraphics[width=\linewidth,keepaspectratio=true]{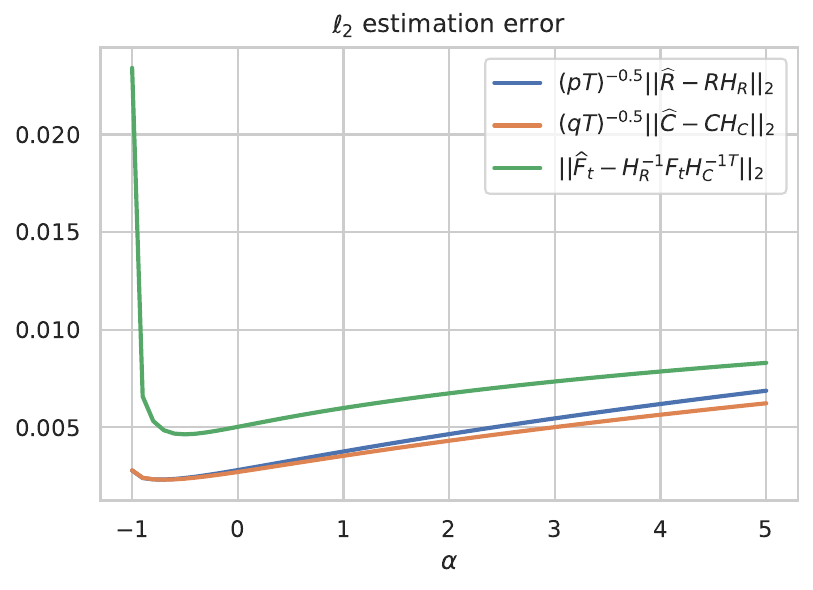}
    \end{subfigure}

    \begin{subfigure}[b]{.3\textwidth}
        \centering
        \caption{Setting \ref{case:I}}
        \includegraphics[width=\linewidth,keepaspectratio=true]{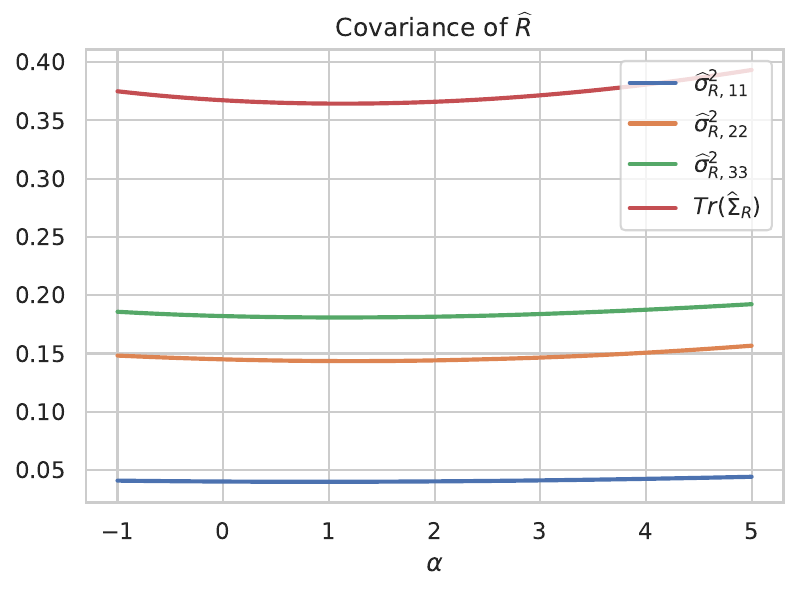}
    \end{subfigure}
    \begin{subfigure}[b]{.3\textwidth}
        \centering
        \caption{Setting \ref{case:I}}
        \includegraphics[width=\linewidth,keepaspectratio=true]{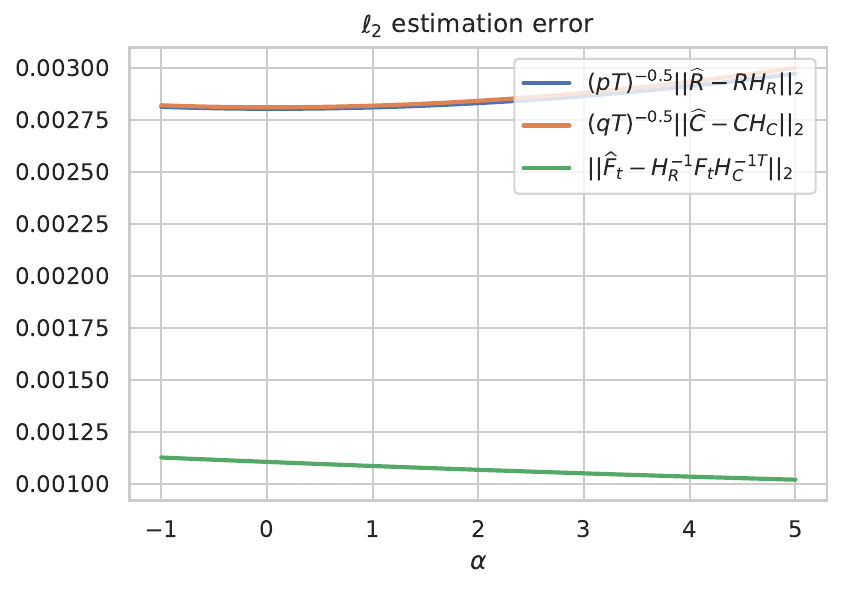}
    \end{subfigure}
    \caption{Covariance of $\sqrt{qT} \left( \hat\bR_{1 \cdot} - \bH_R^\top \bR_{1 \cdot} \right)$ and $\ell_2$ estimation error versus different value of $\alpha$'s in $[-1, 5]$ with a step-size of $0.1$.
        Subplots (a) and (b) are under the Setting (IV) where $\bmu_F \ne \bzero$.
        Subplots (c) and (d) are under Setting (I) where $\bmu_F=0$.
        Values plotted are means of 200 repetitions.} \label{fig:alpha}
\end{figure}

Figure \ref{fig:alpha} (c) and (d) are simulated under Setting (I) where $\bF_t$ has zero mean.
As expected the value of $\alpha$ does not make much difference in the estimators' properties.

\section{Applications}
\label{sec:application}

\subsection{Example 1: Multinational Macroeconomic Indices} \label{sec:application_macro_indices}

In this section, we apply our estimation method to the multinational macroeconomic indices data set used in \cite{chen2019constrained}.
The data set is collected from OECD.
It contains 10 quarterly macroeconomic indices of 14 countries from 1990.Q2 to 2016.Q4  for 107 quarters.
Thus, we have $T = 107$ and $p_1 \times p_2 = 14 \times 10$ matrix-valued time series.
The countries include United States, Canada, New Zealand, Australia, Norway, Ireland, Denmark, United Kingdom, Finland, Sweden, France, Netherlands, Austria and Germany.
The indices cover four major groups, namely production (P:TIEC, P:TM, GDP), consumer price (CPI:Food, CPI:Ener, CPI:Tot), money market (IR:Long, IR:3-Mon), and international trade (IT:Ex, IT:Im).
Each original univariate time series is transformed by taking the first or second difference or logarithm to satisfy the mixing condition in Assumption \ref{assume:alpha-mixing}.
See Table
10 
in Appendix
D 
for detailed descriptions of the data set and transformations.
Figure 16 
in Appendix D shows the transformed time series of macroeconomic indicators of multiple countries.
It is obvious that there exist some similar patterns among time series in the same row or column.

We apply the $\alpha$-PCA proposed in Section \ref{sec:aggr} for different $\alpha$ in the range of $[-1,5]$ with step size $0.1$ on the OECD data set.
We use the ratio-based method in \eqref{eqn:eigen-ratio} as well as the scree plots to estimate the number of latent dimensions.
Using the scree plot to select the minimal number of dimensions that explain at least 80 percent of the variance of $\hat \bM$, we get that $\hat k, \hat r = 4, 6$.
While the ratio based method gives $\hat k, \hat r = 1,2$.
Due to the dominance of the largest factors and weak signal in real data, the estimate by \eqref{eqn:eigen-ratio} tends to be much smaller than the one given by the scree plot.
However, for the purpose of presenting and analyzing some example loading matrix estimates, we will illustrate with latent dimensions $(k, r) = (4, 4)$.

Letting $\hat\bSigma_R= p^{-1}\sum_{i=1}^p \hat\bSigma_{R_i}$ and $\bSigma_C= q^{-1}\sum_{j=1}^q \hat\bSigma_{C_j}$, we plot the traces $\Tr{\hat\bSigma_R}$ and $\Tr{\hat\bSigma_C}$ versus different values of $\alpha$'s in Figure \ref{fig:oecd-traces-alpha}.
The minimizing $\alpha$'s for $\Tr{\hat\bSigma_R}$ and $\Tr{\hat\bSigma_C}$ are $\hat\alpha_R=0.5$ and $\hat\alpha_C=0.6$, respectively.
Note that the proposed estimation method supports using different values of $\alpha_R$ and $\alpha_C$, since the estimation of $\bR$ and $\bC$ are decoupled and the $\alpha$ can be any finite given scalars in $[-1,\infty)$.
Since $\hat\alpha_R$ and $\hat\alpha_C$ are close, we choose $\alpha=0.55$ in the middle for a simple illustration.
To illustrate the interpretation of model \eqref{eqn:mfm} in the real data set, we first present and analyze the loading matrices estimated by $\alpha$-PCA with $\alpha=0.55$ .
Figures \ref{fig:oecd_chen_eig} presents the eigenvalues and the eigen-ratios of $\paran{\hat\bM_{R}, \hat\bM_{C}}$ calculated according to \eqref{eqn:hat-MR-aggr} and \eqref{eqn:hat-MC-aggr} with with $\alpha=0.55$.

\begin{figure}[htpb!]
    \centering
    \includegraphics[width=.4\linewidth,height=\textheight,keepaspectratio=true]{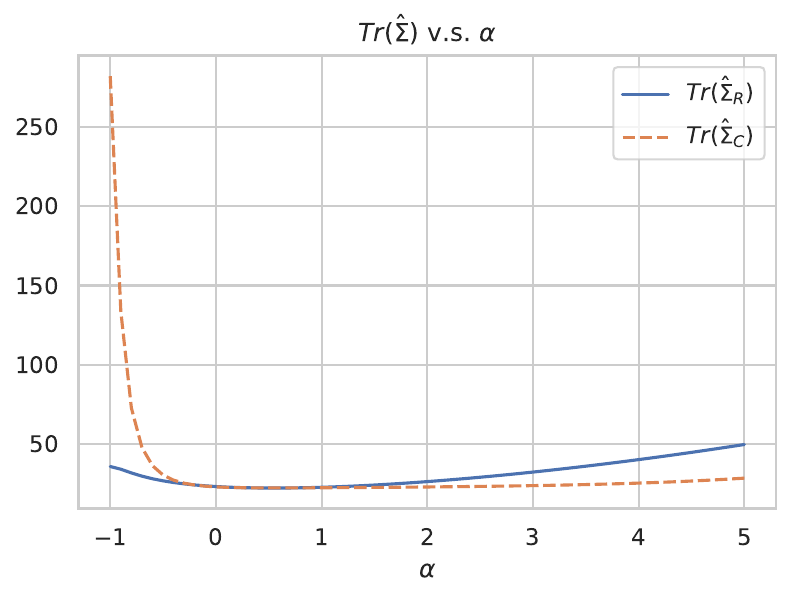}
    \caption{Traces of covariance $\Tr{\hat\bSigma_R}$ and $\Tr{\hat\bSigma_C}$ versus different values of $\alpha$'s in the range of $[-1,5]$ with step size $0.1$.
        The minimizing $\alpha$'s for $\Tr{\hat\bSigma_R}$ and $\Tr{\hat\bSigma_R}$ are $0.5$ and $0.6$, respectively. }
    \label{fig:oecd-traces-alpha}
\end{figure}

\begin{figure}[ht!]
     \centering
    \begin{subfigure}[b]{.3\textwidth}
        \centering
        \includegraphics[width=\linewidth,keepaspectratio=true]{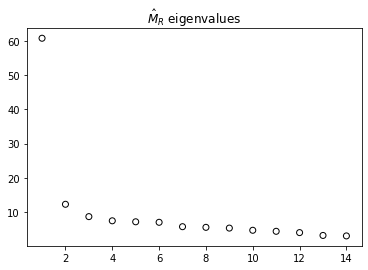}
    \end{subfigure}
    \begin{subfigure}[b]{.3\textwidth}
        \centering
        \includegraphics[width=\linewidth,keepaspectratio=true]{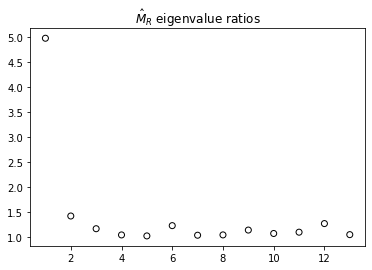}
    \end{subfigure}

    \begin{subfigure}[b]{.3\textwidth}
        \centering
        \includegraphics[width=\linewidth,keepaspectratio=true]{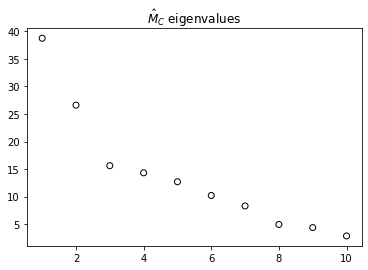}
    \end{subfigure}
    \begin{subfigure}[b]{.3\textwidth}
        \centering
        \includegraphics[width=\linewidth,keepaspectratio=true]{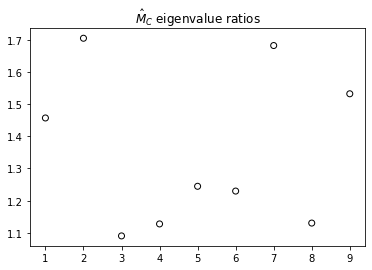}
    \end{subfigure}
    \caption{Eigenvalues and ratios of $\hat \bM_{R}$ and $\hat \bM_{C}$ using the OECD data, using $\alpha$-PCA with $\alpha=0.55$. }
    \label{fig:oecd_chen_eig}
\end{figure}

From these $\hat\bM$ with $(k,r) = (4, 4)$, we calculate loading matrices $\hat\bR_{\alpha}, \hat\bC_{\alpha}$ and $\hat\bR_{AC}, \hat\bC_{AC}$ for $\alpha$-PCA and AC-PCA, respectively.
Table \ref{table:Q_rot} shows estimates of the row and column loading matrices.
They are normalized so that the norm of each column is one, VARIMAX-rotated to reveal a clear structure, and scaled and rounded for ease of display.

We can interpret the latent structure of the global macro-economy by analyzing the estimated row and column loading matrices.
Specifically, from  pair of $\hat\bR_{\alpha,rot}$ and $\hat\bC_{\alpha,rot}$ or pair $\hat\bR_{AC,rot}$ and $\hat\bC_{AC,rot}$ we can group (clustering) some of countries or macroeconomic indices based on their loading matrices.
Using row loading matrices, three groups can easily be formed: Group 1: (USA, CAN), Group 2: (NZL, AUS), Group 3: (FRA, NLD, AUT, DEU). In this  example, USA and CAN both load heavily on row 3 of $\hat\bR_{\alpha,rot}$ and $\hat\bR_{AC, rot}$, but lightly on all other rows, NZL and AUS both load heavily only on row 2 of $\hat\bR_{\alpha,rot}$ and $\hat\bR_{AC, rot}$, and FRA, NLD, AUT, DEU all load the most on rows 1.
This analysis can reveal what countries have stronger correlations in their macroeconomic features.
Interestingly, loading matrices estimated by both methods tend to suggest similar groupings.

From the column loading matrices, we can form groups 1(CPI:Food, CPI: Tot, CPI: Ener), 2:(IR:Long, IR: 3-Mon), 3:(P:TIEC, P:TM, GDP), 4: (IT:Ex, IT:Im) for both $\hat\bC_{\alpha,rot}$ and $\hat\bC_{AC,rot}$.
We can also infer the meaning of each latent column factor from the column loading matrices.
Take $\hat\bC_{\alpha,rot}$ for example, groups 1,2, 3, 4 load most heavily on the 2nd, 4th, 3rd and 1st rows, respectively.
Thus, the 2nd, 4th, 3rd and 1st column factors can be interpreted as factors that are related to consumer price, money market, production, and international trade, respectively.
The results are consistent with our prior knowledge of these macroeconomic indices, where groups 1-4 correspond to the major groups we previously introduced.
Corresponding rotated factor series are plotted in Figure \ref{fig:oecd-fhat}.

\begin{table}[htpb!]
    \centering
    \begin{subtable}[c]{\textwidth}
        \resizebox{\linewidth}{!}{
            \begin{tabular}{cc|cccccccccccccc}
                \hline
                Model & Row & USA & CAN & NZL & AUS & NOR & IRL & DNK & GBR & FIN & SWE & FRA & NLD & AUT & DEU \\ \hline
                \multirow{4}{*}{$\hat{\bR}_{\alpha, rot}$} & 1 & 1 & 0  & -1 & -1 & 2 & 2 & 3 &2 & 3 & 3 & 4 & 4 & 4 & 4 \\
                & 2 & 1 & 0 & 6 & 6 & 2 & 2 & 2 & 3 & 1 &2 & 0 &0 & -1 & -1\\
                & 3 & 6 & 7 &1 &0 & -1 &-1 & -1 &0 & 0 & -2 & 0 & 1 & 0 & 0 \\
                & 4 & 0 & 0 & 0 & 1 & 8 & -5 & -1 & -1 & 0 & 0 & -1 & 1 & 0 & 0 \\ \hline
                 \multirow{4}{*}{$\hat{\bR}_{AC, rot}$} & 1 &-1 & 2 & 1 & -1 & -1 &-1 &-2 &-4 & -3 & -4 & -4 & -4 & -4 & -4 \\
                & 2 & 2 & -1 & 5 & 5 & 1 &5 &3 & 2 &-1 &1 &1 &0 & 0 & 0 \\
                & 3 & 7 & 7 &1 & 1 & -1 &-2 &-1 &0 & 1 & 0 & 0 & 0 & 0 & -1 \\
                & 4 & 1 & -1 & -1 & -2 & -9 & 3 & 0 & 0 & 0 & -1 & 1 & -1 & 0 & 0 \\ \hline
            \end{tabular}
        }%
        \vspace{2mm}
    \end{subtable}
    \begin{subtable}[c]{\textwidth}
        \resizebox{\linewidth}{!}{
            \begin{tabular}{cc|cccccccccc}
                \hline
                Model & Row & CPI:Food & CPI:Tot & CPI:Ener & IR:Long & IR:3-Mon & P:TIEC & P:TM & GDP & IT:Ex & IT:Im \\ \hline
                \multirow{4}{*}{$\hat{\bC}_{\alpha, rot}$} & 1 & 0& 0 & 0 & 0 & 0 & 6 & 6 & 5 & 0 & 0 \\ & 2
                & 6 & 5 & 7 & 0 & 1 & 1 & 0 & -1 & 0 &0\\
                & 3 & -2 & 1 & 0 &0 &0 & 0 & 0 & 0 & 7 & 7 \\
                & 4 & -1 & 1 & 0 & 7 & 7 &-1 &0 &1 & 0 & 0 \\ \hline
                \multirow{4}{*}{$\hat{\bC}_{AC, rot}$} & 1 & -2 &4 & 1 & 1 & -1 & 0 & 0 & 0 & 6 & 6 \\
                & 2 & 6 & 3 &7 & -1 & 1 & 0 & 0 & -1 &-1 &0 \\
                & 3 & -1 & 0 & 1 &0 &0 & -6 & -6 & -6 & 0 & 0 \\
                & 4 & 0 & -1 & 0 & -8 & -6 & 1 &0 & -1 & 0 & 0 \\ \hline
            \end{tabular}
        }
    \end{subtable}
    \caption{Estimations of row and column loading matrices (VARIMAX rotated) of $\alpha$-PCA (subscripted by $\alpha$) and AC-PCA (subscripted by $AC$) with $\alpha = 0.55$ for multinational macroeconomic indices. The loadings matrix are multiplied by $10$ and rounded to integers for ease in display.}
    \label{table:Q_rot}
\end{table}

\begin{figure}[ht!]
    \centering
    \begin{subfigure}[b]{.45\textwidth}
        \centering
        \caption{$\alpha$-PCA, $\alpha=0.55$}
        \includegraphics[width=\linewidth,keepaspectratio=true]{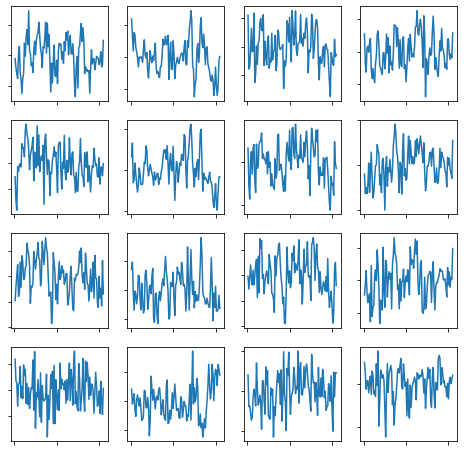}
    \end{subfigure}
    \hspace{4ex}
    \begin{subfigure}[b]{.45\textwidth}
        \centering
        \caption{AC-PCA}
        \includegraphics[width=\linewidth,keepaspectratio=true]{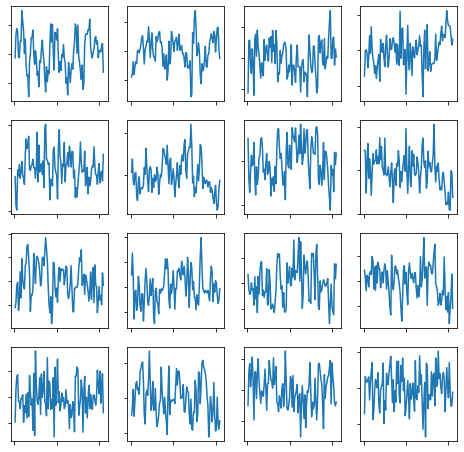}
    \end{subfigure}
    \caption{Plots of rotated $\hat\bF_t \in \RR^{4\times 4}$ estimated by $\alpha$-PCA, $\alpha=0.55$ and AC-PCA, respectively.
    The rotation corresponds to the VARIMAX rotation of $\hat\bR$ and $\hat\bC$ in Table \ref{table:Q_rot}.
    According to the weights in Table  \ref{table:Q_rot}, the 1st - 4th columns correspond to the important components of GDP, CPI, international trade and interest rate, respectively.
    }   \label{fig:oecd-fhat}
\end{figure}

Next, we illustrate choosing best alpha values based on prediction errors.
Specifically, we use 10-fold cross validation (CV) to compare the performance of $\alpha$-PCA with different $\alpha$ in the range of $[-1,2]$ with AC-PCA (with lag factor $h_0 = 2$).
We divide the entire time span into 10 sections and choose each of them as testing data. With time series data, the training data may contain two disconnected time spans.
For AC-PCA, in the case of disconnected $n$ time spans we calculate matrices $\hat \bM_R^{(1)} \dots \hat \bM_R^{(n)}$ according to \eqref{eqn:hat-MR-aggr} over each time span separately.
The matrix $\hat \bM_R$ is re-defined as the sum of $\sum_{i=1}^n{\hat \bM_R^{(i)}}$.
Loading matrices and latent dimensions are estimated from this newly defined $\hat \bM_R$ with procedures in Section \ref{sec:estimation}.
We define out of sample $R^2$ on a testing set of size $N$ as
\begin{equation} \label{eqn:out-of-sample-R2}
\text{\rm out of sample } R^2 \defeq 1- \frac{\sum_{t=1}^N \norm{\bY_t - \hat\bY_t}_F^2}{\sum_{t=1}^N \norm{\bY_t - \bar\bY}_F^2},
\end{equation}
where $\bar\bY=\frac{1}{N}\sum_{t=1}^N \bY_t$ and $\hat\bY_t = \hat\bR\hat\bR^\top\bY_t\hat\bC\hat\bC^\top$.
The denominator is the baseline total sum of squares (TSS) from approximating $\bY_t$ by the sample mean $\bar\bY$.
The nominator represent the residual sum of squares (RSS) from approximating $\bY_t$ by $\hat\bY_t$.
The total sum of squares (TSS) averaged over the 10-fold CV on the testing set is 1451.35, computed using sample average as estimator.
Figure \ref{fig:appl-alpha} (a) shows the out of sample $R^2$ versus different values of $\alpha$ for models with different chosen latent dimensions.
According the metric of maximizing the out of sample $R^2$, the best value of $\alpha$ is 0.4 for latent dimensions $(4, 4)$.
The values of the out of sample $R^2$ are reported in Table \ref{table:macro_indices_ss_outofsample_10CV} for models for the maximizing $\alpha$ and $\alpha = -1, 0, 1$ with different chosen latent dimensions.
All reported values are the averages over the 10-fold CV.
Evidently, the proposed estimation procedure with all chosen values of $\alpha$ performs better than AC-PCA at each chosen $(k,r)$ pair, even though we do not account for temporal dependence.
This implies that the contemporaneous covariance should not be discarded even for the time series data.

\begin{table}[ht!]
    \centering
   \begin{tabular}{c|c|cccccc}
       \hline\cr
       \multicolumn{2}{c|}{\diagbox
       {Method}{$(k,r)$}} & (6,5) & (5,5) & (4,5) & (4,4) & (3,4) & (3,3) \\ \hline
       \multirow{4}{*}{$\alpha$-PCA} & $\alpha=-1$ & 0.465 & 0.422 & 0.392 & 0.310 & 0.296 & 0.159 \\ \cline{2-8}
       & $\alpha=0$ & 0.553 & 0.515 & 0.478 & 0.418 & 0.387 & 0.320 \\ \cline{2-8}
       & $\alpha=1$ & 0.551 & 0.506 & 0.481 & 0.420 & 0.383 & 0.324 \\ \cline{2-8}
       & $\alpha_{opt}$ & \begin{tabular}[c]{@{}c@{}}0.556\\ (0.3)\end{tabular} & \begin{tabular}[c]{@{}c@{}}0.516\\ (-0.2)\end{tabular} & \begin{tabular}[c]{@{}c@{}}0.486\\ (0.7)\end{tabular} & \begin{tabular}[c]{@{}c@{}}0.424\\ (0.4)\end{tabular} & \begin{tabular}[c]{@{}c@{}}0.391\\ (0.3)\end{tabular} & \begin{tabular}[c]{@{}c@{}}0.328\\ (0.2)\end{tabular} \\ \hline
       \multicolumn{2}{l|}{AC-PCA} & 0.429 & 0.393 & 0.354 & 0.248 & 0.216 & 0.092 \\ \hline
   \end{tabular}
    \caption{Results of 10-fold CV of out-of-sample performance for the multinational macroeconomic indexes.
        The numbers shown are average over the cross validation.
        The numbers in parentheses on the line of $\alpha_{opt}$ are the values of $\alpha$'s maximizing the out-of-sample $R^2$. }
    \label{table:macro_indices_ss_outofsample_10CV}
\end{table}

\subsection{Example 2: Image data sets} \label{sec:application_image}

An important category of matrix variables is the 2-D gray-scale image data.
One gray-scale image is represented as a single matrix $\bY_t$, with each element corresponding to one image pixel. The values in the matrix represent intensities within some range.
In this section, we apply our method to two real-world image data sets:
\begin{itemize}
    \item ORL\footnote{\url{http://www.uk.research.att.com/facedatabase.html}} is a well-known dataset for face recognition \citep{samaria1994parameterisation}.
    It contains the face images of $40$ persons, for a total of $400$ images.
    The size of the images is $92 \times 112$.
    \item USPS \footnote{\url{http://www-stat-class.stanford.edu/~tibs/ElemStatLearn/data.html}} is an image data set consisting of $9298$ handwritten digits of ``0'' through ``9''.
    We use a subset of USPS.
    This subset contains $300$ images for each digit, for a total of $3000$ images. The resolution of the images is $16 \times 16$.
\end{itemize}

The estimation of the low-rank signal part $\hat\bR\hat\bF_t\hat\bC^\top$ in \eqref{eqn:mfm} can be viewed as a compressed reconstruction of the original image.
In the signal processing literatures, the goodness of approximation can be measure by the \textit{Root Mean Squared Reconstruction Error (RMSRE)} which is basically the square root of the mean residual sum of squares (RSS).
To be consistent with Section \ref{sec:application_macro_indices}, we use the ratio between RSS and TSS in the empirical evaluation of our method with different values of $\alpha$.
Figure \ref{fig:appl-alpha} (b) and (c) show, respectively for ORL and USPS, the plots of RSS/TSS versus different values of $\alpha$ for models with different chosen latent dimensions.
The small error suggests of dimensionality reduction from the original image $\bY_t$ to the new representation $\bF_t$ is effective.

\begin{figure}[htpb!]
    \centering
    \begin{subfigure}[b]{.3\textwidth}
        \centering
        \caption{Out-of sample $R^2$, OECD}
        \includegraphics[width=\linewidth,keepaspectratio=true]{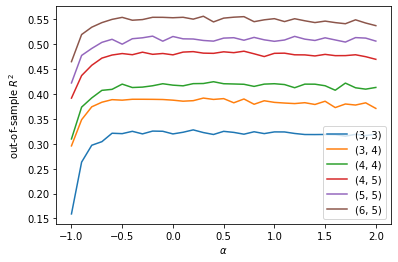}
    \end{subfigure}
    \begin{subfigure}[b]{.3\textwidth}
        \centering
        \caption{RSS/TSS, ORL}
        \includegraphics[width=\linewidth,keepaspectratio=true]{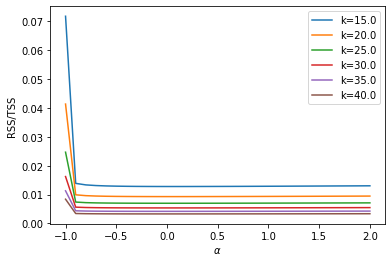}
    \end{subfigure}
    \begin{subfigure}[b]{.3\textwidth}
        \centering
        \caption{RSS/TSS, USPS}
        \includegraphics[width=\linewidth,keepaspectratio=true]{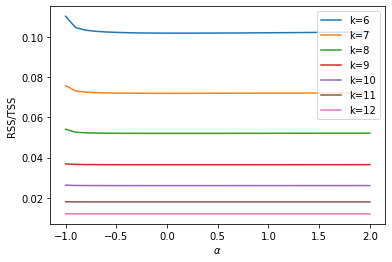}
    \end{subfigure}
    \caption{Choosing $\alpha$ by cross validation using different metrics.
        The values of $\alpha$ are from -1 to 2 with step size of 0.1.
        The out-of sample $R^2$ is defined in \eqref{eqn:out-of-sample-R2}.
    } \label{fig:appl-alpha}
\end{figure}

Tables \ref{table:orl-rec-error} and \ref{table:usps-rec-error} report values of the percentage of RSS/TSS for selected $\alpha$ and the optimal $\alpha$.
The optimal $\alpha$ is $0$ or is very close to $\alpha=0$ and the their differences of RSS/TSS are negligible ($10^{-6}$).  This is in line with our theoretical result.
The method with $\alpha=-1$ produces the largest errors.
The different between $\alpha=1$ and $2$ are small while both are a little worse than $\alpha=0$.

\begin{table}[htpb!]
    \centering
    \begin{tabular}{c|cccccc}
        \hline
        $\alpha$ & $15\times 15$ & $20\times 20$ & $25\times 25$ & $30 \times 30$ & $35 \times 35$ & $40 \times 40$ \\ \hline
        -1 & 7.1721 & 4.1329 & 2.4675 & 1.6245 & 1.1310 & 0.8411 \\
       \rowcolor[HTML]{EFEFEF} 0 & 1.2799 & 0.9308 & 0.7004 & 0.5390 & 0.4206 & 0.3315 \\
        1 & 1.2888 & 0.9372 & 0.7050 & 0.5428 & 0.4236 & 0.3339 \\
        2 & 1.3045 & 0.9489 & 0.7139 & 0.5497 & 0.4290 & 0.3383 \\\hline
        $\min\paran{RSS/TSS}$ & 1.2798 & 0.9307 & 0.7004 & 0.5390 & 0.4206 & 0.3315 \\
        $\alpha_{opt}$ & 0.1 & 0.1 & 0 & 0 & 0 & 0 \\\hline
    \end{tabular}%
    \caption{Percentage of the ORL reconstruction RSS/TSS ($\%$).
    The columns correspond to different values of latent dimension $k \times k$. }
    \label{table:orl-rec-error}
\end{table}

\begin{table}[htpb!]
    \centering
    \begin{tabular}{c|ccccccc}
        \hline
        $\alpha$ & $6\times 6$ & $7\times 7$ & $8\times 8$ & $9\times 9$ & $10\times 10$ & $11\times 11$ & $12\times 12$ \\ \hline
        -1 & 11.0150 & 7.5755 & 5.4047 & 3.6838 & 2.6256 & 1.8049 & 1.2059 \\
        \rowcolor[HTML]{EFEFEF} 0 & 10.1758 & 7.1874 & 5.1994 & 3.6413 & 2.6048 & 1.7944 & 1.1996 \\
        1 & 10.1945 & 7.1967 & 5.2027 & 3.6427 & 2.6055 & 1.7946 & 1.1997 \\
        2 & 10.2317 & 7.2124 & 5.2090 & 3.6458 & 2.6072 & 1.7954 & 1.2001 \\\hline
        $\min\paran{RSS/TSS}$ & 10.1749 & 7.1874 & 5.1993 & 3.6412 & 2.6047 & 1.7943 & 1.995 \\
        $\alpha_{opt}$ & 0.1 & 0.1  & 0.1  & 0.1 & 0.2 & 0.2 & 0.2 \\\hline
    \end{tabular}%
    \caption{Percentage of the USPS reconstruction RSS/TSS ($\%$).
        The columns correspond to different values of latent dimension $k \times k$. }
    \label{table:usps-rec-error}
\end{table}

Figure 17
and 18 
in Appendix E 
show images of $10$ different persons from the ORL and USPS data sets, respectively.
We use $15 \times 15$ latent dimension for the ORL faces and $9 \times 9$ for the USPS digits.
The $10$ images in the first row are the original images from the data set.
The $10$ images in the second row are the ones compressed by our method with $\alpha=-1$, which is the same as the $(2D)^2PCA$ algorithm.
The third, forth, and fifth rows corresponds to our method with $\alpha=0$, $1$, and $2$, respectively.
We observe visually that the proposed method with $\alpha=0$ produces the best compression result, while the method with $\alpha=-1$ performs the worst.
The differences between $\alpha=1$ and $2$ are very small and not visually detectable.

\section{Conclusion}  \label{sec:conclusion}

This paper studies the problem of estimating unknown parameters and latent factors from matrix-variate factor model.
Specifically, we preserve the structure of matrix-variate data and investigate theoretical properties in the setting that the each dimension of the matrix-variates ($p \times q$) is comparable to or greater than the number of observations ($T$).
The estimation procedure aggregates information of both first and second moments.
It incorporates traditional PCA based methods as a special case.
We derive some inferential theory concerning the estimators, including the rate of convergence and limiting distributions.
In contrast to previous estimation methods based on auto-covariance, we use more information based on the contemporary data and are also able to consistently estimate the loading matrices and factor matrices for uncorrelated matrix observations when the auto-covariance method can not.
In addition, our results are obtained under very general conditions that allow for correlations across time, rows and columns.